\documentclass[smallextended]{svjour3}
\usepackage[centertags,intlimits]{amsmath}
\usepackage{amsfonts}
\usepackage{graphicx}
\allowdisplaybreaks[2]
\numberwithin{equation}{section}

\providecommand{\abs}[1]{\lvert #1\rvert}

\newcommand{\nc}{\newcommand}
\nc{\vb}{\mathbf{v}}
\nc{\bx}{\mathbf{x}}
\nc{\by}{\mathbf{y}}
\nc{\bz}{\mathbf{z}}
\nc{\bu}{\mathbf{u}}
\nc{\bv}{\mathbf{v}}
\nc{\ba}{\mathbf{a}}
\nc{\bs}{\mathbf{s}}
\nc{\bq}{\mathbf{q}}
\nc{\bd}{\mathbf{d}}
\nc{\bb}{\mathbf{b}}
\nc{\bc}{\mathbf{c}}
\nc{\bi}{\mathbf{i}}
\nc{\bfr}{\mathbf{r}}
\nc{\bP}{\mathbf{P}}
\nc{\bQ}{\mathbf{Q}}
\nc{\R}{\mathbb R}
\nc{\N}{\mathbb N}
\nc{\bbC}{\mathbb C}
\nc{\D}{\mathbb D}
\nc{\Z}{\mathbb Z}
\nc{\F}{\mathbf F}
\nc{\bbS}{\mathbb S}
\nc{\bE}{\mathbf E}
\nc{\B}{\cal B}
\nc{\br}{\bigr}
\nc{\bl}{\bigl}
\nc{\Bl}{\Bigl}
\nc{\Br}{\Bigr}
\nc{\ind}[1]{\,\mathbf{1}_{\{#1\}}\,}

\author{Anatolii A. Puhalskii\\
University of Colorado Denver and\\
 Institute for Problems in\\ Information
Transmission, Moscow
}
\date{}

\title{On the   $M_t/M_t/K_t+M_t$ queue in heavy traffic}
\institute{ \at
           University of Colorado Denver\\ and Institute for Problems\\
           in Information Transmission, Moscow\\
              \email{anatolii.puhalskii@ucdenver.edu}
}

\authorrunning{A.A.Puhalskii}

\begin{document}
\maketitle
\sloppy
\begin{abstract}
The focus of this paper is on the asymptotics of 
large-time numbers of customers    in time-periodic
Markovian many-server
queues with customer abandonment in heavy traffic. Limit 
theorems are obtained  for the periodic
 number-of-customers processes under the fluid and
diffusion scalings.
Other results concern limits 
for  general time-dependent queues  and  
 for   time-homogeneous queues in steady state.
 \end{abstract}

\section{Introduction}
Many-server queues with customer abandonment
have  been the subject of extensive research,  the primary
motivation coming from modelling call centres, see, e.g., Garnett, Mandelbaum, and Reiman
\cite{GarManRei02}, Whitt \cite{Whi04,Whi05}, Zeltyn and Mandelbaum
\cite{ZelMan05}, and references therein.  
Those papers testify to the importance of the asymptotics where both
the arrival  rate and the number of servers tend to
infinity, their ratio being maintained,  whereas the service and abandonment
rates are kept fixed. Most studied is the case of
  Poisson
arrival processes and exponential service and abandonment times where the
arrival, service, and abandonment rates, and the number of servers 
do not vary with time. Fleming, Simon, and Stolyar \cite{FleSimSto94},
assuming   critical loading, 
 obtain  
diffusion-scale limit  theorems for the stationary number of customers.
  Garnett, Mandelbaum, and Reiman
\cite{GarManRei02}, also for the critical load,
 derive fluid- and diffusion-scale limits   for the
number-of-customers
and virtual-waiting-time
processes, and for the stationary distributions of those processes. 
Their other results  are concerned  with limits for  the
stationary fractions of abandoning
customers and of customers who have to wait in the queue, as well as with
computing expectations of functions of the waiting time.
Similar asymptotics for 
the overloaded case are obtained in  Whitt \cite{Whi04}, who
assumes a finite waiting room,
 and Talreja
and Whitt \cite{TalWhi09}. In addition,
 Whitt \cite{Whi04} provides insight into the case where the number of
 servers is much greater than the abandonment rate. 
Talreja and Whitt \cite{TalWhi09}  also give a
proof of  the virtual-waiting-time-process limit 
for the critically loaded queue. 
 The Markovian
assumptions are relaxed in Zeltyn and Mandelbaum
\cite{ZelMan05} who study steady-state waiting times. 
A  general framework of Markovian stochastic processing systems with
time-varying rates is studied by Mandelbaum, Massey, and Reiman
\cite{ManMasRei98} who 
obtain fluid- and diffusion-scale  limits  for the number-of-customers
 processes. 
They  do not require  certain loading conditions to hold.
The application to many-server queues with abandonment
is explored in  a series of papers
 by Mandelbaum, Massey, Reiman,
 and Stolyar who consider time-varying rates, allow the possibility of
 retrials, and incorporate virtual-waiting-time processes,   see, e.g.,
Mandelbaum, Massey, Reiman,
Rider, and Stolyar \cite{ManMasReiRidSto02} and references therein.

The purpose of this paper is 
  a study of 
Markovian many-server queues with customer abandonment in heavy traffic
 for  a time-periodic case
where the arrival, abandonment, and service rates, and the
 number of servers can be modelled as jointly periodic functions of
 time. Under those hypotheses, the large-time  distributions of the
 numbers of customers are
 periodic.
The main result of this paper states that 
the large-time distributions of the properly scaled and normalised
numbers of customers  converge
 to the periodic distribution of a
limiting diffusion process which arises as a particular case of the results of
 Mandelbaum, Massey, and Reiman
\cite{ManMasRei98}. The convergence of the periodic one-dimensional
distributions is further extended to  convergence of the periodic
processes.
 The method of proof consists in establishing
convergence of the number-of-customers processes and  in checking the tightness of the
stationary  distributions of  embedded discrete-time Markov chains. 
That makes the results of
 Mandelbaum, Massey, and Reiman
\cite{ManMasRei98}  essential.
Unfortunately, the
 proofs there contain flaws, as specified in Remark \ref{re:flaw} below.
Therefore, before embarking on the analysis of the large-time
behaviour, I provide a separate  proof of the heavy traffic
convergence in distribution of the number-of-customers processes 
in  many-server queues with
time-varying rates and abandonment. Unlike the proof of
 Mandelbaum, Massey, and Reiman \cite{ManMasRei98}, who invoke  the strong
approximation techniques,  the proof here relies on
 the  martingale theory of weak convergence which seems to be more
 suitable for this sort of result. An overview of the general approach
 and the related literature
 as well as a heavy-traffic analysis of 
 the time-homogeneous many-server queue
 with abandonment in critical loading
can
 be found in Whitt, Pang, and Talreja \cite{Whi07}. 
The part dealing with tightness relies on
 bounds on the first and second moments of the numbers of customers 
which are uniform over time and  may be of interest in
their own right. 
The approach used can be traced back to Liptser and Shiryayev
\cite[Theorem 8.3.2]{lipshir} and
 Smorodinskii \cite{MR838384}.
Along with the
 application to the
 periodic case, I use 
the convergence of the processes and the moment bounds
in order  to establish convergence  of the stationary
number of customers   in the time-homogeneous case for all three possible loads:
supercritical, critical, and 
 subcritical. On the one hand, this provides 
a unified treatment of and a different perspective on the results of 
Fleming, Simon, and Stolyar \cite{FleSimSto94}, Garnett, Mandelbaum, and Reiman
\cite{GarManRei02}, and Whitt \cite{Whi04} on the limits of the
stationary number of customers. On the other hand, not only are the limits for the
one-dimensional stationary distributions  obtained, but also limits
for the stationary versions of the corresponding processes.
In addition, it is shown that allowing
 the abandonment and service rates to depend on the scaling parameter gives rise
 to  extra  terms in the limit distributions.

The rest of the paper is organised as follows. In
Section~\ref{sec:heavy-traffic-limit}, the  results on the convergence
of the number-of-customers  processes are stated and
proved (Theorem \ref{the:fluid} concerns the fluid scaling and
Theorem~\ref{the:diff} concerns the  diffusion scaling).
 Section~\ref{sec:conv-peri-stat} is concerned with the
 periodic case, the main results being presented in
 Theorem~\ref{the:conv_det_period} and Theorem~\ref{the:conv_period_diff}.
In Section~\ref{sec:stationary-solutions}, the time-homogeneous case
is considered, see Theorem~\ref{the:sta_flui} and
Theorem~\ref{the:stat}. 
The moment
bounds   are
relegated to the appendix, see Lemma~\ref{le:bound}. 
This paper is an expanded and corrected version of 
Puhalskii \cite{Puh07}.

\paragraph{Notation and conventions.}
The set of real numbers
 is denoted by $\R$, the set of nonnegative reals is denoted by
$\R_+$,    the set of natural numbers is denoted by $\N$,
and the set of whole numbers is denoted by $\mathbb{Z}_+$. For real
 numbers $x$ and $y$, $x\wedge y=\min(x,y)$, $x\vee y=\max(x,y)$,
 $x^+=x\vee0$, and $\lfloor x\rfloor$ denotes the integer part;
 $\mathbf{1}_A$ denotes the indicator function of set $A$.
A real-valued function $(f(t),\,t\in\R_+)$ is
said to be strongly
majorised by a real-valued function $(g(t),\,t\in\R_+)$ if  $f(0)\le g(0)$
and the function
$(g(t)-f(t),\,t\in\R_+)$ is nondecreasing. With a slight abuse of
notation, this relationship is
denoted by $f(t)\prec g(t)$\,.
I will say that a function $(f(t),\,t\in\R_+)$ 
is $T$-periodic, where $T>0$, if $f(t+T)=f(t)$ for
all $t$ and that a stochastic process $(X(t),\,t\in\R_+)$ is $T$-periodic
if the distributions of $(X(t+T),\,t\in\R_+)$ and 
of $(X(t),\,t\in\R_+)$ coincide.

The space of rightcontinuous $\R$-valued functions on $\R_+$ with
lefthand limits is denoted by $\D(\R_+,\R)$ and is endowed with
Skorohod's $J_1$-topology and the Borel $\sigma$-algebra.
 For a function
$(x_t,\,t\in\R_+)$ from $\D(\R_+,\R)$, $x_{t-}$ represents the lefthand limit
 at $t$ with the convention that $x_{0-}=0$
 and $\Delta x_t=x_t-x_{t-}$\,. 
All stochastic processes are assumed to have trajectories from and
are considered as random elements of  $\D(\R_+,\R)$\,. Convergence in distribution in
$\D(\R_+,\R)$ has a
standard meaning. The predictable quadratic variation process of a
locally square integrable martingale $(M_t,\,t\in\R_+)$ is denoted by 
$(\langle M\rangle_t,\,t\in\R_+)$\,. (For more
background in weak convergence theory and martingale theory,  the  reader is referred to 
Jacod
and Shiryaev \cite{jacshir} and Liptser and Shiryayev \cite{lipshir}.)
The random entities encountered in the article are defined on a complete
probability space $(\Omega,\mathcal{F},\mathbf{P})$\,.
\label{sec:introduction}

\section{Convergence of the number-of-customers processes}
\label{sec:heavy-traffic-limit}

I will  consider a sequence of $M_t/M_t/K_t+M_t$ queues indexed by
$n\in\N$. 
The $n$th queue is fed by a
Poisson process of customers of rate $\lambda^n_t$ at time $t$. 
The customers are served by one
of the $K^n_t$ servers on a FCFS basis. 
They  may  abandon the
queue after an exponentially distributed time with parameter
$\theta^n_t$ at time $t$\,. More specifically, conditioned on the arrival time
$\tau$, the distribution function of the time until abandonment is
given by $1-\exp(-\int_0^t\theta^n_{s+\tau}\,ds),\,t\in\R_+$\,.
Similarly, the service times of the customers are exponential with
parameter $\mu^n_t$ at time $t$\,.
A customer in service may be relegated to the head of the queue before
her service is complete if
the server serving the customer becomes unavailable because $K^n_t$ decreases.
In that case, the customer starts service from scratch  the
next time she enters service. (The specific policy used for choosing
the server to be removed is inconsequential for the results obtained below.)

The functions $\lambda^n_t$, $\mu^n_t$, and $\theta^n_t$ are assumed to be
$\R_+$-valued locally integrable  functions, i.e., 
$\int_0^t \lambda^n_s\,ds<\infty$, 
$\int_0^t \mu^n_s\,ds<\infty$, and $\int_0^t \theta^n_s\,ds<\infty$ for
all $t\in\R_+$\,. The
functions $K^n_t$ are $\R_+$-valued and Lebesgue measurable.
 The  number of customers present at time $0$, the arrival
process, the service times, and the abandonment times are mutually
independent.

Let $A^n_t$ denote  the number of customer arrivals by time $t$. As
mentioned, the
process $A^n=(A^n_t,\,t\in\R_+)$ is
a Poisson process with time-varying  rate $\lambda^n_t$\,. Customer
abandonment will be modelled via independent Poisson processes
$R^{n,i}=(R^{n,i}_t,\,t\in\R_+),\,i\in\N,$ of rate $\theta^n_t$ at
time $t$ and customer service  will
be modelled via independent Poisson processes $B^{n,i}=(B^{n,i}_t,\,t\in\R_+),\,i\in\N,$ of
rate $\mu^n_t$ at time $t$. Let $Q^n_t$ represent the  number of customers present at
time $t$.
The evolution of the customer population is modelled by the following
equation
\begin{equation}
  \label{eq:1}
  Q^n_t=Q^n_0+A^n_t-\sum_{i=1}^\infty \int_0^t \ind{Q^n_{s-}\ge
K^n_{s-}+  i}\,dR^{n,i}_s- \sum_{i=1}^\infty
\int_0^t\ind{Q^n_{s-}\wedge K^n_{s-}\ge i}\,dB^{n,i}_s\,.
\end{equation}
For an explanation, the third term on the right represents the number
of customers who have
abandoned the queue by time $t$ and the last term represents the number of
 service completions by time $t$\,.
Informally, all customers in service are arranged in 
order and the $i$th
customer is assigned  Poisson process $B^{n,i}$\,. A jump of
$B^{n,i}$ triggers a service completion. Once that occurs, the
customers in service are reordered and are assigned possibly
different  processes $B^{n,i}$ so that there are no gaps in the
sequence of the processes $B^{n,i}$ being used.  Due to the memoryless property
 of the exponential distribution, this reassignment does not
affect the service time distributions.
The indicator function in the
fourth term equals one if and only if a jump of $B^{n,i}$ triggers a
service completion, so the jump of that term at time $t$ equals
$\sum_{i=1}^{Q^n_{t-}\wedge K^n_{t-}}\Delta B^{n,i}_t$ which is the
number of the  processes $B^{n,i}$ ``being used'' that jump at time $t$\,.
(One may want to keep in mind that at most one of these processes
jumps at any given time a.s.)
The processes $R^{n,i}$  are associated with
 the abandonment process in a similar fashion.
(Equation (\ref{eq:1}) also applies to nonFCFS
service disciplines so long as service is performed when there are
customers present. Besides, the customers whose service is interrupted
due to a lack of servers do not have to be put
necessarily at the  head of the queue. The purpose of
those assumptions is to
make the set-up more specific.)

Equation (\ref{eq:1}) has a unique strong solution whose trajectories belong to $\D(\R_+,\R_+)$ which can be shown by applying an
iterative argument on the jump times of the Poisson processes.
Let me also note that the infinite series, in fact, represent finite sums.

Let processes $M^{n,A}=(M^{n,A}_t,\,t\in\R_+)$,
$M^{n,R}=(M^{n,R}_t,\,t\in\R_+)$, and $M^{n,B}=(M^{n,B}_t,\,t\in\R_+)$
be defined by the relations
\begin{subequations}
  \begin{align}
      \label{eq:2a}
  M^{n,A}_t&=A^n_t-\int_0^t\lambda^n_s \,ds,\\
\label{eq:2b}
  M^{n,R}_t&=\sum_{i=1}^\infty \int_0^t \ind{Q^n_{s-}\ge
K^n_{s-}+  i}\,dR^{n,i}_s-\int_0^t\theta^n_s (Q^n_s-K^n_s)^+\,ds,\\
\label{eq:2c}
M^{n,B}_t&=\sum_{i=1}^\infty \int_0^t \ind{Q^n_{s-}\wedge K^n_{s-}\ge
  i}\,dB^{n,i}_s-\int_0^t \mu^n_s\,( Q^n_s\wedge K^n_s)\,ds.
\end{align}
\end{subequations}
By 
 (\ref{eq:1}), (\ref{eq:2a}), (\ref{eq:2b}), and (\ref{eq:2c}),
\begin{equation}
  \label{eq:2}
  Q^n_t=Q^n_0+
\int_0^t\lambda^n_s\,ds-\int_0^t\theta^n_s
\bl(Q^n_s-
K^n_s\br)^+\,ds-
\int_0^t\mu^n_s\, (Q^n_s\wedge K^n_s) \,ds+
\,M^{n,A}_t-M^{n,R}_t-M^{n,B}_t\,.
\end{equation}
The  following martingale characterisation plays a key role in
subsequent developments. Let $\mathcal{F}^n_t$ denote the completion with respect to
$\mathbf{P}$ of the $\sigma$-algebra generated by the random variables $Q^n_0$,
$A^n_s$, $B^{n,i}_s$, and $R^{n,i}_s$, where $s\le t$ and $i\in\N$\,. The associated filtration is denoted 
by $\mathbf{F}^n$ so that
$\mathbf{F}^n=(\mathcal{F}^n_t,\,t\in\R_+)$\,.
It may be  worth noting  that $Q^n_t$ is $\mathcal{F}^n_t$-measurable.
\begin{lemma}
  \label{le:mart}

The processes $M^{n,A}$,
$M^{n,R}$, and $M^{n,B}$
are $\mathbf{F}^n$-locally square integrable martingales with respective
predictable quadratic variation processes

\begin{subequations}
  \begin{align}
      \label{eq:3a}
  \langle M^{n,A}\rangle_t&=\int_0^t\lambda^n_s\,ds,\\
\label{eq:3b}
 \langle M^{n,R}\rangle_t&=\int_0^t\theta^n_s\, (Q^n_s-K^n_s)^+\,ds,\\
\label{eq:3c}
\langle M^{n,B}\rangle_t&=\int_0^t\mu^n_s\,( Q^n_s\wedge K^n_s)\,ds.
\end{align}
\end{subequations}
In addition,  these locally square integrable
martingales are pairwise orthogonal, i.e., their mutual
predictable characteristics are equal to zero:
\begin{equation}
  \label{eq:4}
   \langle M^{n,A},M^{n,R}\rangle_t=
\langle M^{n,A},M^{n,B}\rangle_t=\langle M^{n,R},M^{n,B}\rangle_t=0\,.
\end{equation}
\end{lemma}
\begin{proof}
According to the definition,  $M^{n,A}$ is an $\mathbf{F}^n$-martingale. Since
$\mathbf{E}(M^{n,A}_t)^2=\int_0^t\lambda^n_s\,ds<\infty$, it is a locally square integrable
martingale. One easily checks that
$\bl((M^{n,A}_t)^2-\int_0^t\lambda^n_s\,ds,\,t\in\R_+\br) $ is a
martingale. Similarly, the processes 
$(H^{n,R,i}_t,\,t\in\R_+)$ and 
$(H^{n,B,i}_t,\,t\in\R_+)$, where $H^{n,R,i}_t=R^{n,i}_t-\int_0^t\theta^n_s\,ds$ and 
$H^{n,B,i}_t=B^{n,i}_t-\int_0^t\mu^n_s\,ds$, 
are pairwise orthogonal $\mathbf{F}^n$-locally square integrable
martingales with predictable quadratic variation processes
$\langle H^{n,R,i}\rangle_t=\int_0^t\theta^n_s\,ds$ and 
$\langle H^{n,B,i}\rangle_t=\int_0^t\mu^n_s\,ds$\,,respectively.

  One can write by (\ref{eq:2b}) and (\ref{eq:2c}) that
  \begin{subequations}
    \begin{align*}
  M^{n,R}_t&=\sum_{i=1}^\infty M^{n,R,i}_t\intertext{and}
M^{n,B}_t&=\sum_{i=1}^\infty M^{n,B,i}_t ,
\end{align*}
  \end{subequations}
where 
\begin{align*}
M^{n,R,i}_t =\int_0^t \ind{Q^n_{s-}\ge 
K^n_{s-}+  i}\,(dR^{n,i}_s-\theta^n_s \,ds)\intertext{and}
M^{n,B,i}_t =\int_0^t \ind{Q^n_{s-}\wedge K^n_{s-}\ge
  i}\,(dB^{n,i}_s- \mu^n_s\, ds)  \,.
\end{align*}
As stochastic integrals with respect to locally square integrable
martingales,
the processes $M^{n,R,i}=(M^{n,R,i}_t,\,t\in\R_+)$ and
$M^{n,B,i}=(M^{n,B,i}_t,\,t\in\R_+)$ are  $\mathbf{F}^n$-locally square integrable
martingales. Their mutual predictable characteristics
are given by 
$\langle M^{n,R,i},M^{n,R,i'}\rangle_t=\int_0^t \ind{Q^n_{s-}\ge 
K^n_{s-}+  i\vee i'}\,d\langle H^{n,R,i},H^{n,R,i'}\rangle_s$\,,
$\langle M^{n,B,i},M^{n,B,i'}\rangle_t=\int_0^t \ind{Q^n_{s-}\wedge
K^n_{s-}\ge  i\vee i'}\,d\langle H^{n,B,i},H^{n,B,i'}\rangle_s$\,, and
$\langle M^{n,R,i},M^{n,B,i'}\rangle_t=\int_0^t \ind{Q^n_{s-}\ge 
K^n_{s-}+  i}\ind{Q^n_{s-}\wedge
K^n_{s-}\ge   i'}\,d\langle H^{n,R,i},H^{n,B,i'}\rangle_s$\,.
Therefore, the locally square integrable martingales 
 $M^{n,R,i}$ and
$M^{n,B,i}$ are  pairwise orthogonal 
 with respective predictable quadratic variation processes 
$(\int_0^t \ind{Q^n_{s}\ge
K^n_{s}+  i}\,\theta^n_s \,ds,\,t\in\R_+)$ and
$(\int_0^t \ind{Q^n_{s}\wedge K^n_{s}\ge
 i}\,\mu^n_s \,ds,\,t\in\R_+)$\,. The stopping times
$\tau^n_k=\inf\{t\in\R_+:\,Q^n_t\ge k\}$, where $k\in\N$, are common
localising times for these locally square integrable martingales and
$\sum_{i=1}^\infty \mathbf{E}(M^{n,R,i}_{t\wedge \tau^n_k})^2<\infty$ and
$\sum_{i=1}^\infty \mathbf{E}(M^{n,B,i}_{t\wedge
  \tau^n_k})^2<\infty$\,. It follows that, when stopped at $\tau^n_k$, the  processes 
$M^{n,R}$ and $ M^{n,B}$ are  square integrable martingales, so they
are locally square integrable martingales with respective predictable
quadratic variation
processes $(\sum_{i=1}^\infty \int_0^t \ind{Q^n_{s}\ge
K^n_{s}+  i}\,\theta^n_s \,ds,\,t\in\R_+)$ and $(\sum_{i=1}^\infty \int_0^t \ind{Q^n_{s}\wedge
K^n_{s}\ge  i}\,\mu^n_s \,ds,\,t\in\R_+)$\,. The fact that 
the locally square integrable martingales $M^{n,A}$,
$M^{n,R}$, and $M^{n,B}$ are pairwise orthogonal follows since those
processes have $\mathbf{P}$-a.s. pairwise disjoint jumps.
\end{proof}
The next theorem establishes a fluid-scale limit.
In the rest of the paper, I will assume as  fixed
$\R_+$-valued locally integrable  functions 
$(\lambda_t,\,t\in\R_+)$,  $(\mu_t,\,t\in\R_+)$, and
$(\theta_t,\,t\in\R_+)$, and an $\R_+$-valued Lebesgue
measurable function $(\kappa_t,\,t\in\R_+)$\,. Given an $\R_+$-valued
random variable $q_0$, 
let  $q_t$ be defined by the equation
\begin{equation}
  \label{eq:14}
    q_t=q_0+\int_0^t\lambda_s\,ds 
-\int_0^t\theta_s (q_s-\kappa_s)^+\,ds-\int_0^t\mu_s\,( q_s\wedge \kappa_s)\,ds.
\end{equation}
The Lipshitz continuity of $(x-\kappa_s)^+$ and of 
$x\wedge \kappa_s$ in $x$ ensures  that  the  equation  has a unique
  solution.
\begin{theorem}
  \label{the:fluid}
Suppose  that, as $n\to\infty$,  $\int_0^t\lambda^n_s/n\,ds\to\int_0^t
\lambda_s\,ds$  
for all $t$, 
that  $\mu^n_t\to\mu_t$  
uniformly on bounded intervals, 
that  $\theta^n_t\to\theta_t$  
uniformly on bounded intervals, 
 and   that $K^n_t/n\to \kappa_t$ for all $t$\,.
If the 
random variables $Q^n_0/n$ converge in distribution to a random
variable
 $q_0$ as
$n\to\infty$,
then the processes $(Q^n_t/n,\,t\in\R_+)$ converge in distribution in
$\D(\R_+,\R)$ to the process $(q_t,\,t\in\R_+)$\,.
In particular, if $q_0$ is deterministic, then
 for all $L>0$ and $\epsilon>0$,
\begin{equation*}
  \lim_{n\to\infty}\mathbf{P}(\sup_{t\in[0,L]}
\abs{\frac{Q^n_t}{n}-q_t}>\epsilon)=0\,.
\end{equation*}
\end{theorem}
\begin{proof}
Let me first assume that $q_0$ is deterministic so that the $Q^n_0/n$
converge to $q_0$ in probability.
I prove that
\begin{equation}
  \label{eq:3}
  \lim_{n\to\infty}\mathbf{P}\bl(\frac{1}{n}\,
\sup_{t\in[0,L]}\abs{M^{n,i}_t}>\epsilon)=0
\end{equation}
for $i=A,R,B$, where $L>0$ and $\epsilon>0$ are otherwise arbitrary. The
L\'englart-Rebolledo inequality, see, e.g., Liptser and Shiryayev
\cite[Theorem 1.9.3]{lipshir},
  implies that it suffices to prove that, for $t\in\R_+$,
\begin{equation}
  \label{eq:5}
  \lim_{n\to\infty}\mathbf{P}\bl(\frac{1}{n^2}\,\langle M^{n,i}\rangle_t>\epsilon)=0\,.
\end{equation}
The validity of (\ref{eq:5}) for $i=A$  follows by
(\ref{eq:3a})  and the hypothesis that
$\int_0^t\bl(\lambda^n_s/n- \lambda_s\br)\,ds\to0$\,. As a
consequence of \eqref{eq:3} for $i=A$, I have that the $A^n_t/n$ converge
in probability to $\int_0^t\lambda_s\,ds$ as $n\to\infty$ uniformly on
bounded intervals\,.
In order to establish (\ref{eq:5}) for $i=R$, I note that by
(\ref{eq:1})
$Q^n_t/n\le Q^n_0/n+A^n_t/n$. Since the latter quantities converge in
probability uniformly on bounded intervals to $q_0+\int_0^t\lambda_s \,ds$ as $n\to\infty$, it follows that 
$\limsup_{n\to\infty}
\mathbf{P}(\sup_{s\in[0,t]}Q^n_s/n>q_0+\int_0^t\lambda_s \,ds+1)\le
\limsup_{n\to\infty}
\mathbf{P}(Q^n_0/n+A^n_t/n>q_0+\int_0^t\lambda_s \,ds+1)
=0$\,. If $n$ is such that
$\sup_{s\in[0,t]}\abs{\theta^n_s-\theta_s}\le 1$, then  by (\ref{eq:3b}),
$\mathbf{P}(\langle M^{n,R}\rangle_t/n^2>\epsilon)\le
\mathbf{P}(\sup_{s\in[0,t]}(Q^n_s/n)
\int_0^t(\theta_s+1)\,ds>n\epsilon)$\,,
which implies (\ref{eq:5}) for $i=R$.
The case  $i=B$ is treated similarly.
The limits in (\ref{eq:3}) have been proved.

By (\ref{eq:2}) and
(\ref{eq:14}),
\begin{multline*}
  \abs{\frac{Q^n_t}{n}-q_t}\le
  \abs{\frac{Q^n_0}{n}-q_0}+\abs{\int_0^t\frac{\lambda^n_s}{n}\,ds-\int_0^t\lambda_s
  \,ds}+
\int_0^t (\theta^n_s+\mu^n_s)\abs{\frac{Q^n_s}{n}-q_s}\,ds\\
+\abs{\int_0^t\mu^n_s\,\bl( q_s\wedge\frac{K^n_s}{n}\br)\,ds
-\int_0^t\mu_s\, (q_s\wedge \kappa_s)\,ds}
+\abs{\int_0^t\theta^n_s (q_s-\frac{K^n_s}{n})^+\,ds
-\int_0^t\theta_s (q_s-\kappa_s)^+\,ds}\\
+\frac{1}{n}\,\abs{M^{n,A}_t}+
\frac{1}{n}\,\abs{M^{n,R}_t}+\frac{1}{n}\,\abs{M^{n,B}_t}\,.
\end{multline*}
By 
Gronwall's
 inequality, see, e.g., p.498 in Ethier and Kurtz \cite{EthKur86},
for $L>0$,
\begin{multline*}
  \sup_{t\in[0,L]}\abs{\frac{Q^n_t}{n}-q_t}\le
\bl(  \abs{\frac{Q^n_0}{n}-q_0}+
\sup_{t\in[0,L]}\abs{\int_0^t\frac{\lambda^n_s}{n}\,ds-\int_0^t\lambda_s\,ds}\\+
\sup_{t\in[0,L]}\abs{\int_0^t\mu^n_s\,\bl( q_s\wedge\frac{K^n_s}{n}\br)\,ds
-\int_0^t\mu_s\, (q_s\wedge \kappa_s)\,ds}\\
+\sup_{t\in[0,L]}\abs{\int_0^t\theta^n_s (q_s-\frac{K^n_s}{n})^+\,ds
-\int_0^t\theta_s (q_s-\kappa_s)^+\,ds}
+\frac{1}{n}\,\sup_{t\in[0,L]}\abs{ M^{n,A}_t}\\+
\frac{1}{n}\,\sup_{t\in[0,L]}\abs{M^{n,R}_t}+\frac{1}{n}\,
\sup_{t\in[0,L]}\abs{M^{n,B}_t}\br)e^{\int_0^L(\theta^n_t+\mu^n_t)\,dt}\,.
\end{multline*}
By (\ref{eq:3}) and the
hypotheses, the righthand side tends in probability to zero as $n\to\infty$.

Suppose now that $q_0$ is random.
Let $\mathbf{\Theta}_x^n$ denote the distribution on  $\D(\R_+,\R)$ of
$(Q^n_t/n,\,t\in\R_+)$ provided that $Q^n_0/n=x\in \Sigma^n$, where $\Sigma^n=
\{0,1/n,2/n,\ldots\}$,
 let $\mathbf{\Xi}^n$ represent the 
distribution of $Q^n_0/n$, and let $\mathbf{\Xi}$ represent the 
distribution of $q_0$\,. By the independence assumptions, 
 it suffices to prove that,
for a bounded continuous function $f$ on $\D(\R_+,\R)$, 
\begin{equation}
  \label{eq:13a}
    \lim_{n\to\infty}\int_{\D(\R_+,\R)\times \Sigma^n}
  f(z)\,\mathbf{\Theta}^n_{x}(dz)\,\mathbf{\Xi}^n(dx)=\int_{ \R_+}
 f(q(x))\,\mathbf{\Xi}(dx)\,,
\end{equation}
where $q(x)=(q_t,\,t\in\R_+)$ is defined by (\ref{eq:14}) with $q_0=x$\,.
 By the part just proved, if $x^n\to x$, where $x^n\in\Sigma^n$,
 and $x\in\R_+$,
then the $\mathbf{\Theta}_{x^n}^n$ weakly converge to the Dirac
measure at $q(x)$, so
\begin{equation}
  \label{eq:12a}
    \lim_{n\to\infty}\int_{\D(\R_+,\R)} f(z)\,\mathbf{\Theta}^n_{x^n}(dz)=
f(q(x))\,.
\end{equation}
Given $x\in\R_+$, let $g(x)= f(q(x))$ and 
$g^n(x)=\int_{\D(\R_+,\R)} f(z)\,\mathbf{\Theta}^n_{r(x)}(dz)$, where $r(x)$ represents the
element of $\Sigma^n$ which is closest to $x$ on the left. By
\eqref{eq:12a}, if $x^n\to x$, where $x^n\in\R_+$, then
$g^n(x^n)\to g(x)$\,. The weak convergence of the $\mathbf{\Xi}^n$ to $\mathbf{\Xi}$ implies
that
\begin{equation*}
  \lim_{n\to\infty}\int_{\R_+}g^n(x)\,\mathbf{\Xi}^n(dx)=\int_{\R_+}g(x)\,\mathbf{\Xi}(dx).
\end{equation*}
Since $\mathbf{\Xi}^n(\Sigma^n)=1$, I have that
$\int_{\R_+}g^n(x)\,\mathbf{\Xi}^n(dx)=\int_{\D(\R_+,\R)\times \Sigma^n}
  f(z)\,\mathbf{\Theta}^n_{x}(dz)\,\mathbf{\Xi}^n(dx)$, so
 \eqref{eq:13a} follows\,.

\end{proof}
\begin{corollary}
\label{co:mart}
Suppose that   the hypotheses of  Theorem~\ref{the:fluid} hold where
$q_0$ is deterministic. Then the processes
  $M^{n,A}/\sqrt{n}=(M^{n,A}_t/\sqrt{n},\,t\in\R_+)$,
  $M^{n,R}/\sqrt{n}=(M^{n,R}_t/\sqrt{n},\,t\in\R_+)$, 
and $M^{n,B}/\sqrt{n}=(M^{n,B}_t/\sqrt{n},\,t\in\R_+)$
jointly  converge in distribution in $\D(\R_+,\R)$ to the respective
  processes $M^A=(M^A_t,\,t\in\R_+)$, $M^R=(M^R_t,\,t\in\R_+)$, and
 $M^B=(M^B_t,\,t\in\R_+)$, defined as follows:
 \begin{align*}
   M^A_t&=\int_0^t\sqrt{\lambda_s}\,dW^A_s,\\
M^R_t&=\int_0^t \sqrt{\theta_s (q_s-\kappa_s)^+}\,dW^R_s,\\
M^B_t&=\int_0^t \sqrt{\mu_s\,(q_s\wedge \kappa_s)}\,dW^B_s,
 \end{align*}
where $W^A=(W^A_t,\,t\in\R_+)$, $W^R=(W^R_t,\,t\in\R_+)$, and
 $W^B=(W^B_t,\,t\in\R_+)$ are independent standard Wiener processes.
\end{corollary}
\begin{proof}
  The processes $M^{n,A}/\sqrt{n}$, $M^{n,R}/\sqrt{n}$, and
  $M^{n,B}/\sqrt{n}$ are $\mathbf{F}^n$-locally square integrable
  martingales. By (\ref{eq:3a}), (\ref{eq:3b}),  (\ref{eq:3c}), and 
(\ref{eq:4}) they are mutually orthogonal and
  their respective predictable quadratic variation processes are given by
  \begin{align*}
  \langle \frac{M^{n,A}}{\sqrt{n}}\rangle_t&=\int_0^t
\frac{\lambda^n_s}{n}\, ds,\\
 \langle\frac{M^{n,R}}{\sqrt{n}}\rangle_t&=\int_0^t \theta^n_s
\bl(\frac{Q^n_s}{n}-\frac{K^n_s}{n}\br)^+\,ds,\\
\langle\frac{M^{n,B}}{\sqrt{n}}\rangle_t&=\int_0^t \mu^n_s\,
\bl(\frac{Q^n_s}{n}\wedge \frac{K^n_s}{n}\br)\,ds.
\end{align*}
By Theorem~\ref{the:fluid} and the hypotheses, 
the random variables  on the right converge in
probability to the functions $\int_0^t\lambda_s\,ds$, 
$\int_0^t\theta_s (q_s-\kappa_s)^+\,ds$, and $\int_0^t \mu_s\,
(q_s\wedge \kappa_s)\,ds$, respectively. (Actually the first convergence is
deterministic.)
I also have by (\ref{eq:2a}), (\ref{eq:2b}), and (\ref{eq:2c}) that
the jumps of the processes $M^{n,A}/\sqrt{n}$, $M^{n,R}/\sqrt{n}$, and
  $M^{n,B}/\sqrt{n}$ are not greater than $1/\sqrt{n}$\,.
The proof is finished by an application of Theorem 7.1.4
 in Liptser and Shiryayev \cite{lipshir}.
\end{proof}
Let me introduce
\begin{subequations}
  \begin{align}
    \label{eq:35}
\alpha^n_t&=    \sqrt{n}\bl(\frac{\lambda^n_t}{n}-\lambda_t\br),\\
    \label{eq:35b}
\beta^n_t&=    \sqrt{n}\bl(\mu^n_t-\mu_t\br),\\
    \label{eq:35c}
\gamma^n_t&=    \sqrt{n}\bl(\theta^n_t-\theta_t\br),\intertext{and}
    \label{eq:35a}
\delta^n_t&=    \sqrt{n}\bl(\frac{K^n_t}{n}-\kappa_t\br)\,.  \end{align}
\end{subequations}
Let  processes $X^n=(X^n_t,\,t\in\R_+)$ be defined by
\begin{equation}
  \label{eq:6}
  X^n_t=\sqrt{n}\bl(\frac{Q^n_t}{n}-q_t\br)\,.  
\end{equation}
 By (\ref{eq:2}) and (\ref{eq:14}),  I can write
  \begin{multline}
          \label{eq:7}
X^n_t=X^n_0+\int_0^t\alpha^n_s\,ds
-\int_0^t\theta^n_s\Bl(\bl(X^n_s+ \sqrt{n} q_s-(\delta^n_s+
\sqrt{n}\kappa_s)\br)^+
-\sqrt{n}(q_s-\kappa_s)^+\Br)\,ds\\-
\int_0^t\mu^n_s\Bl(\bl(X^n_s+ \sqrt{n} q_s\br)\wedge \bl(\delta^n_s+
\sqrt{n}\kappa_s\br)-
\sqrt{n}\,(q_s\wedge \kappa_s)\Br)\,ds
-\int_0^t\gamma^n_s(q_s-\kappa_s)^+\,ds
\\-
\int_0^t\beta^n_s\,(q_s\wedge \kappa_s)\,ds
+\frac{1}{\sqrt{n}}\,M^{n,A}_t-
\frac{1}{\sqrt{n}}\,M^{n,R}_t-\frac{1}{\sqrt{n}}\,M^{n,B}_t\,.    
  \end{multline}
In the rest of the paper,   $(\alpha_t,\,t\in\R_+)$, 
$(\beta_t,\,t\in\R_+)$, and $(\gamma_t,\,t\in\R_+)$ represent
  locally 
integrable functions and 
$(\delta_t,\,t\in\R_+)$ represents a locally bounded Lebesgue measurable
 function.

The following theorem yields a diffusion-scale limit.
\begin{theorem}
  \label{the:diff}
Let the hypotheses  of Theorem~\ref{the:fluid} hold where $q_0\in\R_+$ is
deterministic. 
Suppose   that $\int_0^t\alpha^n_s\,ds
\to\int_0^t \alpha_s\,ds$,
 $\beta^n_t
\to \beta_t$, $\gamma^n_t
\to \gamma_t$, and $\delta^n_t
\to \delta_t$  uniformly on bounded
intervals as $n\to\infty$, 
and that the random variables $X^n_0$ converge in distribution to a random variable $X_0$ as
$n\to\infty$\,. Then the processes $X^n$
 converge in
distribution in $\D(\R_+,\R)$ to the process $X=(X_t\,,t\in\R_+)$ that
is the solution of the equation
\begin{multline*}
  X_t=X_0+\int_0^t\bl(\alpha_s-\gamma_s(q_s-\kappa_s)^+-\beta_s(q_s\wedge\kappa_s)\br)\,ds-\int_0^t\theta_s
\bl(\ind{q_s>\kappa_s}(X_s-\delta_s)
+\ind{q_s=\kappa_s} (X_s-\delta_s)^+\br)\,ds\\-\int_0^t\mu_s
\bl(\ind{q_s<\kappa_s}X_s+\ind{q_s=\kappa_s}(X_s\wedge \delta_s)+
\ind{q_s>\kappa_s}\delta_s\br)\,ds+
\int_0^t\sqrt{\lambda_s+\theta_s (q_s-\kappa_s)^++\mu_s\,(q_s\wedge \kappa_s)}\,dW_s,
\end{multline*}
where $W=(W_t,\,t\in\R_+)$ is a standard Wiener process and $W$ and
$X_0$ are independent.
\end{theorem}
\begin{proof}
  The equation for $X$ has a unique strong solution by the  fact
  that the   infinitesimal drift coefficients are Lipshitz continuous,
  the functions $(\lambda_s,\,s\in\R_+)$, $(\mu_s,\,s\in\R_+)$, 
  $(\theta_s,\,s\in\R_+)$,   $(\alpha_s,\,s\in\R_+)$, 
$(\beta_s,\,s\in\R_+)$, and $(\gamma_s,\,s\in\R_+)$ are  locally 
integrable, and
  the function $(\delta_s,\,s\in\R_+)$ is locally bounded, see, e.g.,
  Ikeda and Watanabe \cite{IkeWat}.

Let me first consider the case of deterministic  $X^n_0$, so
$X_0^n=x^n\in S^n$, where  $S^n$ represents 
the set  of numbers of the form $\sqrt{n}(m/n-q_0)$ for
$m\in \mathbb{Z}_+$, and $x^n\to x\in\R$ as $n\to\infty$\,.
By \eqref{eq:7},
\begin{multline*}
  \abs{X^n_t}\le\abs{x^n}+ 
\abs{\int_0^t\alpha^n_s \,ds}+
\int_0^t(\theta^n_s+\mu^n_s)\abs{X^n_s}\,ds
+\int_0^t(\theta^n_s+\mu^n_s)\abs{\delta^n_s}\,ds\\
+\int_0^t\abs{\gamma^n_s}q_s\,ds
+\int_0^t\abs{\beta^n_s}\,q_s\,ds+
\frac{1}{\sqrt{n}}\,\abs{M^{n,A}_t}+
\frac{1}{\sqrt{n}}\,\abs{M^{n,R}_t}+\frac{1}{\sqrt{n}}\,\abs{M^{n,B}_t}\,.
\end{multline*}
Gronwall's inequality,
 the hypotheses of Theorem~\ref{the:diff}, and
Corollary~\ref{co:mart} imply that for $L>0$
\begin{equation}
  \label{eq:8}
  \lim_{r\to\infty}\limsup_{n\to\infty}
\mathbf{P}(\sup_{t\in[0,L]}\abs{X^n_t}>r)=0\,.
\end{equation}
Also, for $s\le t$,
\begin{multline*}
  \abs{X^n_t-X^n_s}\le 
\abs{\int_s^t\alpha^n_u\,du}+
\int_s^t(\theta^n_u+\mu^n_u)\abs{X^n_u}\,du
+\int_s^t(\theta^n_u+\mu^n_u)\abs{\delta^n_u}\,du\\+
\int_s^t\abs{\gamma^n_u}q_u\,du
+\int_s^t\abs{\beta^n_u}\,q_u\,du+
\frac{1}{\sqrt{n}}\,\abs{M^{n,A}_t-M^{n,A}_s}+
\frac{1}{\sqrt{n}}\,\abs{M^{n,R}_t-M^{n,R}_s}
+\frac{1}{\sqrt{n}}\,\abs{M^{n,B}_t-M^{n,B}_s}\,.
\end{multline*}
Given $L>0$, $\eta>0$, and $r>0$\,, 
\begin{multline*}
  \mathbf{P}(\sup_{s,t\in[0,L]:\,\abs{s-t}<\delta}\abs{X^n_t-X^n_s}>\eta)\le
\mathbf{P}(\sup_{t\in[0,L]}\abs{X^n_t}>r)\\+
\mathbf{P}\bl(\sup_{s,t\in[0,L]:\,\abs{s-t}<\delta}\bl(\abs{\int_s^t\alpha^n_u\,du}+
r\int_s^t(\theta^n_u+\mu^n_u)\,du
+\int_s^t(\theta^n_u+\mu^n_u)\abs{\delta^n_u}\,du\\+
\int_s^t\abs{\gamma^n_u}q_u\,du
+\int_s^t\abs{\beta^n_u}\,q_u\,du+
\frac{1}{\sqrt{n}}\,\abs{M^{n,A}_t-M^{n,A}_s}\\+
\frac{1}{\sqrt{n}}\,\abs{M^{n,R}_t-M^{n,R}_s}
+\frac{1}{\sqrt{n}}\,\abs{M^{n,B}_t-M^{n,B}_s}\br)>\eta\br)\,.
\end{multline*}
Hence, by  Corollary~\ref{co:mart} and the
hypotheses, 
\begin{multline*}
  \limsup_{n\to\infty}
\mathbf{P}(\sup_{s,t\in[0,L]:\,\abs{s-t}<\delta}\abs{X^n_t-X^n_s}>\eta)\le
  \limsup_{n\to\infty}\mathbf{P}(\sup_{t\in[0,L]}\abs{X^n_t}>r)\\+
\mathbf{P}\bl(\sup_{s,t\in[0,L]:\,\abs{s-t}<\delta}\bl(
\abs{\int_s^t\alpha_u\,du}+
r\int_s^t(\theta_u+\mu_u)\,du
+\int_s^t(\theta_u+\mu_u)\abs{\delta_u}\,du\\+
\int_s^t\abs{\gamma_u}q_u\,du
+\int_s^t\abs{\beta_u}\,q_u\,du+
\abs{M^{A}_t-M^{A}_s}+
\abs{M^{R}_t-M^{R}_s}
+\abs{M^{B}_t-M^{B}_s}\br)>\frac{\eta}{2}\br)
\,.
\end{multline*}
By the continuity of the processes $M^A$, $M^B$, and $M^R$, and
absolute continuity of the Lebesgue integral, the limit of the second
probability on the right, as $\delta\to0$, equals zero, so
\begin{equation*}
\limsup_{\delta\to0}    \limsup_{n\to\infty}
\mathbf{P}(\sup_{s,t\in[0,L]:\,\abs{s-t}<\delta}\abs{X^n_t-X^n_s}>\eta)\le
  \limsup_{n\to\infty}\mathbf{P}(\sup_{t\in[0,L]}\abs{X^n_t}>r)\,.
\end{equation*}
By (\ref{eq:8}), 
the righthand side can be made arbitrarily small by choosing $r$ great enough.
Therefore,
\begin{equation*}
  \lim_{\delta\to0}\limsup_{n\to\infty}
\mathbf{P}(\sup_{s,t\in[0,L]:\,\abs{s-t}<\delta}\abs{X^n_t-X^n_s}>\eta)=0\,.
\end{equation*}
It follows that  the sequence $X^n$ is $\mathbb{C}$-tight, i.e., it is tight for
convergence in distribution in $\D(\R_+,\R)$, and all limit
points are continuous-path processes. Let $\tilde{X}=(\tilde{X}_t,\,t\in\R_+)$ represent  a
subsequential limit of the $X^n$. 

Let me note  that if a sequence of functions
$(x^n_t,\,t\in\R_+)$ from $\D(\R_+,\R)$ converges for  Skorohod's $J_1$-topology to a continuous function
$(x_t,\,t\in\R_+)$ as $n\to\infty$, then 
\begin{equation}
  \label{eq:15}
  \int_0^t\theta^n_s\bl(\bl(x^n_s+ \sqrt{n} q_s-
(\delta^n_s+
\sqrt{n}\kappa_s)\br)^+
-\sqrt{n}(q_s-\kappa_s)^+\br)\,ds\to \int_0^t\theta_s\bl(\ind{q_s>\kappa_s}
(x_s-\delta_s)+
\ind{q_s=\kappa_s}(x_s-\delta_s)^+\br)\,ds
\end{equation}
and
\begin{equation}
  \label{eq:16}
  \int_0^t\mu^n_s\bl(\bl(x^n_s+ \sqrt{n} q_s\br)\wedge (\delta^n_s+
\sqrt{n}\kappa_s)-
\sqrt{n}\,(q_s\wedge \kappa_s)\br)\,ds\to
\int_0^t\mu_s
\bl(\ind{q_s<\kappa_s}x_s+\ind{q_s=\kappa_s}(x_s\wedge \delta_s)
+\ind{q_s>\kappa_s}\delta_s\br)\,ds\,.
\end{equation}
To see \eqref{eq:15}, one could first note that 
$\theta^n_s\bl(\bl(x^n_s+ \sqrt{n} q_s-
(\delta^n_s+
\sqrt{n}\kappa_s)\br)^+
-\sqrt{n}(q_s-\kappa_s)^+\br)\to \theta_s\bl(\ind{q_s>\kappa_s}
(x_s-\delta_s)+
\ind{q_s=\kappa_s}(x_s-\delta_s)^+\br)$ for each $s$, for if 
$q_s>\kappa_s$, then 
$\bl(x^n_s+ \sqrt{n} q_s-
(\delta^n_s+
\sqrt{n}\kappa_s)\br)^+
-\sqrt{n}(q_s-\kappa_s)^+\br)=
\bl(x^n_s+ \sqrt{n} q_s-
(\delta^n_s+
\sqrt{n}\kappa_s)\br)
-\sqrt{n}(q_s-\kappa_s)\br)=x^n_s-
\delta^n_s$ for all $n$ great  enough, if
$q_s=\kappa_s$, then 
$\bl(x^n_s+ \sqrt{n} q_s-
(\delta^n_s+
\sqrt{n}\kappa_s)\br)^+
-\sqrt{n}(q_s-\kappa_s)^+\br)=(x^n_s-
\delta^n_s)^+$, and 
if 
$q_s<\kappa_s$, then 
$\bl(x^n_s+ \sqrt{n} q_s-
(\delta^n_s+
\sqrt{n}\kappa_s)\br)^+
-\sqrt{n}(q_s-\kappa_s)^+\br)=0$ for all $n$ great  enough.
Since
$\abs{\theta^n_s\bl(\bl(x^n_s+ \sqrt{n} q_s-
(\delta^n_s+
\sqrt{n}\kappa_s)\br)^+
-\sqrt{n}(q_s-\kappa_s)^+\br)}\le
\theta^n_s\abs{x^n_s-\delta^n_s}$, the convergence in \eqref{eq:15}
follows by Lebesgue's dominated convergence theorem. The argument for
\eqref{eq:16} is similar.

On recalling
Corollary~\ref{co:mart}, I conclude from (\ref{eq:7}) and the
continuous mapping principle that 
$\tilde{X}$ must satisfy the equation
\begin{multline*}
  \tilde{X}_t=x+\int_0^t\alpha_s\,ds-\int_0^t\theta_s
\bl(\ind{q_s>\kappa_s}
(\tilde{X}_s-\delta_s)+\ind{q_s=\kappa_s}(\tilde{X}_s-\delta_s)^+\br)\,ds
-\int_0^t\gamma_s(q_s-\kappa_s)^+\,ds\\
-\int_0^t\mu_s
\bl(\ind{q_s<\kappa_s}\tilde{X}_s+\ind{q_s=\kappa_s}(\tilde{X}_s\wedge \delta_s)+
\ind{q_s>\kappa_s}\delta_s\br)\,ds
-\int_0^t\beta_s(q_s\wedge \kappa_s)\,ds\\+
\int_0^t\sqrt{\lambda_s+\theta_s( q_s-\kappa_s)^+
+\mu_s\,(q_s\wedge \kappa_s)}\,d\tilde{W}_s\,,
\end{multline*}
where $(\tilde{W}_t,\,t\in\R_+)$ is a standard Wiener process. Since the latter
equation has a unique  solution, $\tilde{X}$ coincides in law
with $X$, so the $X^n$ converge in distribution to $X$\,.

I will now consider the case of general $X^n_0$\,. The argument is
similar to the one used in the proof  of Theorem~\ref{the:fluid}.
Let $\mathbf{\Phi}_x^n$ denote the distribution on  $\D(\R_+,\R)$ of
$X^n$ provided that $X^n_0=x\in S^n$,
let $\mathbf{\Phi}_x$ denote the distribution on  $\D(\R_+,\R)$ of
$X$ provided that $X_0=x\in\R$\,, let $\mathbf{\Psi}^n$ denote the 
distribution of $X^n_0$, and let $\mathbf{\Psi}$ denote the 
distribution of $X_0$\,. By the independence assumptions, 
 it suffices to prove that,
for a bounded continuous function $f$ on $\D(\R_+,\R)$, 
\begin{equation}
  \label{eq:13}
    \lim_{n\to\infty}\int_{\D(\R_+,\R)\times S^n}
  f(z)\,\mathbf{\Phi}^n_{x}(dz)\,\mathbf{\Psi}^n(dx)=\int_{\D(\R_+,\R)\times \R}
 f(z)\,\mathbf{\Phi}_x(dz)\,\mathbf{\Psi}(dx)\,.
\end{equation}
 By the part just proved, if $x^n\to x$, where $x^n\in S^n$ and $x\in\R$,
then the $\mathbf{\Phi}_{x^n}^n$ weakly converge to $\mathbf{\Phi}_x$\,, so
\begin{equation}
  \label{eq:12}
    \lim_{n\to\infty}\int_{\D(\R_+,\R)} f(z)\,\mathbf{\Phi}^n_{x^n}(dz)=\int_{\D(\R_+,\R)} f(z)\,\mathbf{\Phi}_x(dz)\,.
\end{equation}
Given $x\in\R$, let $g(x)=\int_{\D(\R_+,\R)} f(z)\,\mathbf{\Phi}_{x}(dz)$ and 
$g^n(x)=\int_{\D(\R_+,\R)} f(z)\,\mathbf{\Phi}^n_{r(x)}(dz)$, where $r(x)$ represents the
element of $S^n$ which is closest to $x$ on the left. By
\eqref{eq:12}, if $x^n\to x$, where $x^n\in\R$ and $x\in\R$, then
$g^n(x^n)\to g(x)$\,. The weak convergence of the $\mathbf{\Psi}^n$ to $\mathbf{\Psi}$ implies
that
\begin{equation*}
  \lim_{n\to\infty}\int_{\R}g^n(x)\,\mathbf{\Psi}^n(dx)=\int_{\R}g(x)\,\mathbf{\Psi}(dx).
\end{equation*}
Since $\mathbf{\Psi}^n(S^n)=1$, I have that
$\int_{\R}g^n(x)\,\mathbf{\Psi}^n(dx)=\int_{\D(\R_+,\R)\times S^n}
  f(z)\,\mathbf{\Phi}^n_{x}(dz)\,\mathbf{\Psi}^n(dx)$, so
 \eqref{eq:13} follows\,.
\end{proof}
\begin{remark}
\label{re:flaw}
The assertions of Theorem~\ref{the:fluid} and
  Theorem~\ref{the:diff} are 
contained in  Theorem 2.2 and 
Theorem 2.3, respectively, in 
 Mandelbaum, Massey, and Reiman
\cite{ManMasRei98}, albeit under slightly stronger
hypotheses. However, 
 the proof of 
Lemma 9.3 there depends on the erroneous claim  that if a
 sequence of nonnegative random variables defined on the same
 probability space is tight, then it has a
 finite limit superior a.s. There are also problems with establishing
 the martingale property in the proof of Lemma 9.1.
\end{remark}

\section{Convergence of the  periodic  queue lengths}
\label{sec:conv-peri-stat}
In this section, I will assume that the functions $\lambda^n_t$, $\mu^n_t$,
$\theta^n_t$,  $K^n_t$, $\lambda_t$, $\mu_t$, $\theta_t$, and
$\kappa_t$  are $T$-periodic, where  $T>0$\,.
I will also assume that
\begin{align}
    \label{eq:17}
\int_0^T\lambda_s\,ds&>0\,\intertext{and}
  \label{eq:28}
  \int_0^T(\mu_s\wedge \theta_s)\,ds&>0\,.
\end{align}
In the long term, one expects  a  periodic pattern to emerge
 for the number of customers.  The next lemma confirms that to
be the case.
Let $Q^{n,\ell }=(Q^n_{\ell T+t},\,t\in\R_+)$, where
$\ell \in\Z_+$\,.
The sequence $\{Q^{n,\ell },\,\ell\in\Z_+\}$ is a 
discrete-time homogeneous Markov process with values in $\D(\R_+,\R)$\,.

\begin{lemma}
  \label{le:period_stat}
Suppose that   $\int_0^T(\mu^n_s\wedge \theta^n_s)\,ds>0$\,.
As $\ell \to\infty$, given an arbitrarily distributed $Q^n_0$,
 the sequence of the distributions of the processes
 $Q^{n,\ell }$ converges in the 
distance 
of total variation
in $\D(\R_+,\R)$
 to the distribution of a process 
$\breve{Q}^n=(\breve{Q}^n_{t},\,t\in\R_+)$\,, which
  is a unique $T$-periodic  Markov process with the same transition
probability function as $Q^n$\,.  
The distribution of $(\breve{Q}^n_t,\,t\in\R_+)$  is a stationary
initial distribution for  
 $\{Q^{n,\ell},\,\ell\in\mathbb{Z}_+\}$\,.
\end{lemma}
\begin{remark}
  For the definition of the distance  of total variation, see,
 e.g., p.274 in Jacod and Shiryaev
\cite{jacshir}.
\end{remark}
\begin{proof}
If $\int_0^T\lambda^n_s\,ds=0$, then $A^n_t=0$ for all $t\in\R_+$, so
$0$ is an absorbing state for $Q^n$ and $Q^n_t\to0$ in the distance of total variation as
$t\to\infty$\,, so $\breve{Q}^n_{t}=0$\,.

Suppose that $\int_0^T\lambda^n_s\,ds>0$\,.
Then the sequence $\{Q^n_{\ell T},\,\ell \in\Z_+\}$ is a
time-homogeneous, irreducible and aperiodic discrete-time Markov chain.
One can show as follows that it converges
  in the distance of total variation to
  a unique stationary distribution as $\ell\to\infty$.
It suffices to prove that the chain is positive
 recurrent, which, by Foster's criterion, will follow if, for some $N\in\Z_+$, 
 \begin{equation}
   \label{eq:43}
   \mathbf{E}_xQ^n_T\le x-1\text{ for all }x\in\{N+1,N+2,\ldots\}\,,
 \end{equation}
where $ \mathbf{E}_x$ denotes expectation with respect to the
probability measure $\mathbf{P}_x$ such that
$\mathbf{P}_x(Q^n_0=x)=1$\,, see, e.g., Theorem 11.3.4 on p.265 and
Proposition 13.2.4 on p.319 in  Meyn and Tweedie \cite{MeyTwe93}, or
Theorem 2.2.3 on p.29 in Fayolle, Malyshev, and Men'shikov
\cite{FayMalMen95}.
Since $Q^n_t\le  Q^n_0+A^n_t$ by (\ref{eq:1}),  
 $\mathbf{E}_xQ^n_t \le
x+\int_0^t\lambda^n_s\,ds$\,. By Lemma~\ref{le:mart}, 
the processes $M^{n,A}$, $M^{n,R}$, 
and $M^{n,B}$ are $\mathbf{F}^n$-locally square
integrable martingales under $\mathbf{P}_x$ with
respective predictable quadratic variation processes
$(\int_0^t\lambda^n_s\,ds\,,t\in\R_+)$, 
$(\int_0^t\theta^n_s\, (Q^n_s-K^n_s)^+\,ds,\,t\in\R_+)$, and
$(\int_0^t\mu^n_s\, (Q^n_s\wedge
K^n_s)\,ds,\,t\in\R_+)$\,. Since the latter processes are of finite
expectation,  $\mathbf{E}(M^{n,A}_t)^2<\infty$,
$\mathbf{E}(M^{n,R}_t)^2<\infty$, and
$\mathbf{E}(M^{n,B}_t)^2<\infty$\,. 
In particular,  
the processes $M^{n,A}$,
$M^{n,R}$, and $M^{n,B}$
are $\mathbf{F}^n$-martingales, so by (\ref{eq:2}),
\begin{equation*}
     - \mathbf{E}_xQ^n_t\prec -x+\int_0^t(\theta^n_s\vee\mu^n_s)\mathbf{E}_xQ^n_s\,ds \,,
\end{equation*}which implies by Lemma~\ref{le:maj} that
\begin{equation*}
  \inf_{t\le T}\mathbf{E}_xQ^n_t\ge xe^{-\int_0^T(\mu^n_s\vee\theta^n_s)\,ds}\,.
\end{equation*}
By \eqref{eq:2},
\begin{equation*}
   \mathbf{E}_xQ^n_T\le x+
\int_0^T\lambda^n_s\,ds-\int_0^T(\mu^n_s\wedge\theta^n_s)
\,\mathbf{E}_xQ^n_s\,ds\le x+
\int_0^T\lambda^n_s\,ds-xe^{-\int_0^T(\mu^n_s\vee\theta^n_s)\,ds}\int_0^T\mu^n_s\wedge\theta^n_s
\,ds\,.
\end{equation*}
Therefore, (\ref{eq:43}) holds if 
\begin{equation*}
  N\ge e^{\int_0^T(\mu^n_s\vee\theta^n_s)\,ds}\,\frac{\displaystyle 1+\int_0^T\lambda^n_s\,ds}{\displaystyle\int_0^T\mu^n_s\wedge\theta^n_s
\,ds }\,.
\end{equation*}
Thus, the distributions of $Q^n_{\ell T}$ converge in the distance of total variation to a limit
distribution as $\ell\to\infty$. Since the
transition probability function of $Q^{n,\ell}$ is periodic, 
 the finite-dimensional distributions of $Q^{n,\ell}$ converge in
the distance of total variation to  limit
distributions as $\ell\to\infty$.
Since the Borel and cylindrical $\sigma$-algebras on $\D(\R_+,\R)$
coincide,
it follows that the distributions of the Markov processes $Q^{n,\ell}
$ converge
in the distance of total variation to the distribution of a  process
 $\breve{Q}^n=(\breve{Q}^n_{t},\,t\in\R_+)$, which is a
  $T$-periodic Markov  process with the same
transition probability function as $Q^n$\,. Since the limiting
distribution of the $Q^n_{\ell T}$ is specified uniquely, the
distribution of  $\breve{Q}^n$ is specified uniquely.  Since the sequence 
$\{ \breve{Q}^n_{\ell T}\,,\ell\in\Z_+\}$ is stationary and  the
processes $(\breve{Q}^n_{\ell T+t}\,,t\in\R_+)$ are Markov processes
with the same transition probability function, it follows that the sequence 
$\{(\breve{Q}^n_{\ell T+t}\,,t\in\R_+),\,\ell\in\Z_+\}$ is stationary.
\end{proof}
\begin{remark}
  Note that the sequence $\{Q^{n,\ell}
,\,\ell\in\Z_+\}$ is deterministic once the initial
  condition $(Q^n_{t}\,,t\in\R_+)$ has been chosen.
\end{remark}
My next step is to consider periodic regimes for the deterministic
approximation.
\begin{lemma}
  \label{le:determ_period}
  \begin{enumerate}
  \item 
There exists a unique $q_0\in\R_+$ such that
the function 
$(q_t,\,t\in\R_+)$ defined by
equation (\ref{eq:14})
is $T$-periodic.
 An arbitrary solution 
 converges to this periodic
 solution as $t\to\infty$.
\item If $q_0$ is a random variable such that the process 
$(q_t,\,t\in\R_+)$ is $T$-periodic, then $q_0$ is deterministic and 
has the value specified in part 1.
  \end{enumerate}
\end{lemma}
\begin{proof}

By uniqueness,
 no two  solutions have a point in
common. In particular, if $q_0'>q_0$, then for the corresponding
solutions, $q_t'>q_t$ for all $t\in\R_+$.  Given $q_0$,
there are three
  possibilities: either $ q_T=q_0$, or $q_T>q_0$,
  or $q_T<q_0$\,. If $ q_T=q_0$, then the solution
  starting at $q_0$ is a periodic solution.
Suppose that $ q_T>q_0$\,. Then on taking $q_T$ as
a new initial condition, by periodicity,
$q_{t+T}>q_t$ for all $t\in\R_+$, so $q_{2T}>q_T$.
Continuing on, I obtain an
increasing  sequence
of solutions $(q_{\ell T+t},\,t\in\R_+)$, where
$\ell=0,1,2,\ldots$\,.
By part 1(a) of Lemma~\ref{le:bound} found in the appendix, 
$    \sup_{t\in\R_+}  q_t<\infty\,$,
so there exists a limit
of $q_{\ell T+t}$ as $\ell\to\infty$\,. I denote this limit by
$\breve{q}_t$\,. Since $q_{\ell T}\to\breve{q}_0$ and 
$q_{\ell T+T}\to\breve{q}_T$, $(\breve{q}_t,\,t\in\R_+)$ is a 
$T$-periodic function.
By bounded convergence, it is also a solution.
If $ q_T<q_0$, then $(q_{\ell T+t},\,t\in\R_+)$
is a monotonically decreasing
sequence of functions  converging to a $T$-periodic solution.

To show the uniqueness of a $T$-periodic solution, note that if 
$(\breve{q}_t,\,t\in\R_+)$ is
a $T$-periodic solution, then $\breve{q}_0=\breve{q}_T$, 
so $\int_0^T \lambda_s\,ds=
\int_0^T\theta_s (\breve q_s-\kappa_s)^+\,ds+\int_0^T\mu_s\,( \breve q_s\wedge
\kappa_s)\,ds$\,. Now, if $(q'_t,\,t\in\R_+)$ is a solution with
$q_0'>\breve q_0$, then $q'_t>\breve q_t$ for all $t$, 
so on recalling \eqref{eq:28},
\begin{equation*}
  \int_0^T\theta_s (q'_s-\kappa_s)^+\,ds+\int_0^T\mu_s\, (q'_s\wedge
\kappa_s)\,ds>\int_0^T\theta_s (\breve q_s-\kappa_s)^+\,ds
+\int_0^T\mu_s\, (\breve q_s\wedge
\kappa_s)\,ds=\int_0^T \lambda_s\,ds,
\end{equation*}
which implies that $q_T'<q'_0$. Similarly, if $q_0'<\breve q_0$, then
$q'_T>q'_0$\,.
Thus, $(q'_t,\,t\in\R_+)$ is not $T$-periodic.
Part 1 is proved.

Let 
$(q_t,\,t\in\R_+)$ represent a $T$-periodic process. The  reasoning
used to show the uniqueness of a $T$-periodic solution
shows that 
$\abs{q_T-\breve{q}_0}<\abs{q_0-\breve{q}_0}$ when $q_0\not=\breve{q}_0$\,. 
Since the distributions of 
$\abs{q_T-\breve{q}_0}$ and $\abs{q_0-\breve{q}_0}$ are the same,
 $q_0=\breve{q}_0$ a.s.
\end{proof}
In what follows,  $(\breve{q}_t,\,t\in\R_+)$ represents the $T$-periodic solution of Lemma~\ref{le:determ_period}.
\begin{theorem}
  \label{the:conv_det_period}
Suppose  that, as $n\to\infty$,  $\int_0^t\lambda^n_s/n\,ds\to\int_0^t
\lambda_s\,ds$  
for all $t$, 
that  $\mu^n_t\to\mu_t$  
uniformly on bounded intervals, 
that  $\theta^n_t\to\theta_t$  
uniformly on bounded intervals, 
 and   that $K^n_t/n\to \kappa_t$ for all $t$\,.
Then, for all $\epsilon>0$ and $L>0$,
\begin{equation*}
  \lim_{n\to\infty}\mathbf{P}(\sup_{t\in[0,L]}\abs{\frac{\breve{Q}^n_t}{n}-\breve{q}_t}>\epsilon)=0\,.
\end{equation*}
\end{theorem}
\begin{proof}
Since $\breve{Q}^n_0$ is a limit in distribution of the $Q^n_t$ as
$t\to\infty$, by part 1(b) of 
Lemma~\ref{le:bound} (with $Q^n_0=0$), 
the sequence $\{\breve{Q}^n_0/n,\,n\in\N\}$
is tight. (Note that by (\ref{eq:28}),
$\liminf_{n\to\infty}\int_0^T(\mu^n_s\wedge\theta^n_s)\,ds>0$.) 
 By Theorem~\ref{the:fluid} and Prohorov's theorem, the sequence of processes
$\{(\breve{Q}^n_t/n,\,t\in\R_+),\,n\in\N\}$ is tight and any limit
point $(q_t,\,t\in\R_+)$ is the solution  of (\ref{eq:14})  for a suitable
$q_0$\,. Since the processes $(\breve{Q}^n_t/n,\,t\in\R_+)$ are
$T$-periodic, so is the process $(q_t,\,t\in\R_+)$\,. By
Lemma~\ref{le:determ_period}, $q_t=\breve{q}_t$ a.s., 
which concludes the proof.
\end{proof}
Let 
\begin{equation}
  \label{eq:23}
  \breve{X}^n_t=\sqrt{n}\bl(\frac{\breve{Q}^n_t}{n}-\breve{q}_t\br)\,.
\end{equation}
The process $(\breve{X}^n_{t},\,t\in \R_+)$ is a  $T$-periodic Markov process.
\begin{theorem}
\label{the:conv_period_diff}
Suppose   that $\int_0^t\alpha^n_s\,ds
\to\int_0^t \alpha_s\,ds$, 
 $\beta^n_t
\to \beta_t$, $\gamma^n_t
\to \gamma_t$, and $\delta^n_t
\to \delta_t$  uniformly on bounded
intervals as $n\to\infty$. Then
the processes $(\breve{X}^n_{t},\,t\in \R_+)$ converge in
  distribution as $n\to\infty$ to  process $(\breve{X}_{t},\,t\in
  \R_+)$\,,
which is a unique $T$-periodic Markov process satisfying  the equation 
\begin{multline*}
   \breve{X}_t=\breve{X}_0+\int_0^t\bl(\alpha_s-\gamma_s(\breve{q}_s-\kappa_s)^+-\beta_s(\breve{q}_s\wedge\kappa_s)\br)\,ds-\int_0^t\theta_s
\bl(\ind{\breve{q}_s>\kappa_s}(\breve{X}_s-\delta_s)
+\ind{\breve{q}_s=\kappa_s} (\breve{X}_s-\delta_s)^+\br)\,ds
\\-\int_0^t\mu_s
\bl(\ind{\breve{q}_s<\kappa_s}\breve{X}_s+\ind{\breve{q}_s=\kappa_s}(\breve{X}_s\wedge \delta_s)+
\ind{\breve{q}_s>\kappa_s}\delta_s\br)\,ds+
\int_0^t\sqrt{\lambda_s+\theta_s (\breve{q}_s-\kappa_s)^++\mu_s\,(\breve{q}_s\wedge \kappa_s)}\,d\breve{W}_s,
\end{multline*}
where $(\breve{W}_t,\,t\in\R_+)$ is a standard Wiener process and
$\breve{X}_0$ and $(\breve{W}_t,\,t\in\R_+)$ are independent.
\end{theorem}
\begin{proof}
  By Lemma~\ref{le:period_stat} and Lemma~\ref{le:determ_period},  the processes
$(X^n_{\ell T+t},\,t\in\R_+)$, where $Q^n_0=q_0=0$,  converge in distribution in
$\D(\R_+,\R)$ to 
 $(\breve{X}^n_{t},\,t\in\R_+)$ as $\ell\to\infty$ and the sequence
$\{(\breve{X}^n_{\ell T+t},\,t\in\R_+),\,\ell\in\Z_+\}$ is stationary.
 By part 1(c) of Lemma~\ref{le:bound},
$ \lim_{V\to\infty}\limsup_{n\to\infty}\limsup_{t\to\infty}
\mathbf{P}(\abs{X^n_t}>V)=0.$
Therefore, the sequence $\{\breve{X}^n_{0}\,,n\in\N\}$ is
tight. By Theorem~\ref{the:diff} and Prohorov's theorem, 
the sequence 
 $\{(\breve{X}^n_{t},t\in\R_+)\,,n\in\N\}$ is tight.
Let   $(\breve{X}_t,t\in\R_+)$ represent a limit point of that
sequence for convergence in
distribution in $\D(\R_+,\R)$ as $n\to\infty$\,. As follows by
Theorem~\ref{the:diff}, it satisfies the
equation in the statement  and is a Markov process. In addition,
$\{(\breve{X}_{\ell T+t},\,t\in\R_+),\,\ell\in\Z_+\}$
is a limit point of $\{(\breve{X}^n_{\ell
  T+t},\,t\in\R_+),\,\ell\in\Z_+\}$ as $n\to\infty$ for convergence in
distribution in $\D(\R_+,\R)^{\Z_+}$\,.
Since the sequence
$\{(\breve{X}^n_{\ell T+t},\,t\in\R_+),\,\ell\in\Z_+\}$ is stationary,
so is
the sequence $\{(\breve{X}_{\ell T+t},\,t\in\R_+),\,\ell\in\Z_+\}$.
Hence, $(\breve{X}_t,t\in\R_+)$ is a $T$-periodic Markov process.

The following coupling argument  shows that the distribution of
$(\breve{X}_t,\,t\in\R_+)$ is specified uniquely and is the limit of
the distributions of  processes
$(\tilde{X}_{\ell T+t},\,t\in\R_+)$,  as $\ell\to\infty$\,, where
the process $\tilde{X}=(\tilde{X}_t,\,t\in\R_+)$ is defined by the equation
 \begin{multline*}
  \tilde{X}_t=x+\int_0^t\bl(\alpha_s-\gamma_s(\breve{q}_s-\kappa_s)^+-\beta_s(\breve{q}_s\wedge\kappa_s)\br)\,ds-\int_0^t\theta_s
\bl(\ind{\breve{q}_s>\kappa_s}(\tilde{X}_s-\delta_s)
+\ind{\breve{q}_s=\kappa_s} (\tilde{X}_s-\delta_s)^+\br)\,ds\\-\int_0^t\mu_s
\bl(\ind{\breve{q}_s<\kappa_s}\tilde{X}_s+\ind{\breve{q}_s=\kappa_s}(\tilde{X}_s\wedge \delta_s)+
\ind{\breve{q}_s>\kappa_s}\delta_s\br)\,ds+
\int_0^t\sqrt{\lambda_s+\theta_s (\breve{q}_s-\kappa_s)^++\mu_s\,(\breve{q}_s\wedge \kappa_s)}\,d\tilde{W}_s\,,
\end{multline*}
where $x\in\R$ and
 $(\tilde{W}_t,\,t\in\R_+)$ is
 a standard Wiener process.
 Let me
consider a process
 $\tilde{X}'=(\tilde{X}'_t,\,t\in\R_+)$ which starts at $y\in\R$ and 
  is driven
by the negative of the Wiener process $(\tilde{W}_t,\,t\in\R_+)$ so that
  \begin{multline*}
  \tilde{X}'_t=y+\int_0^t\bl(\alpha_s-\gamma_s(\breve{q}_s-\kappa_s)^+-\beta_s(\breve{q}_s\wedge\kappa_s)\br)\,ds-\int_0^t\theta_s
\bl(\ind{\breve{q}_s>\kappa_s}(\tilde{X}'_s-\delta_s)
+\ind{\breve{q}_s=\kappa_s} (\tilde{X}'_s-\delta_s)^+\br)\,ds\\-\int_0^t\mu_s
\bl(\ind{\breve{q}_s<\kappa_s}\tilde{X}'_s+\ind{\breve{q}_s=\kappa_s}(\tilde{X}'_s\wedge \delta_s)+
\ind{\breve{q}_s>\kappa_s}\delta_s\br)\,ds-
\int_0^t\sqrt{\lambda_s+\theta_s (\breve{q}_s-\kappa_s)^++\mu_s\,(\breve{q}_s\wedge \kappa_s)}\,d\tilde{W}_s\,.
\end{multline*}
Obviously, the distribution of $\tilde{X}'$ is the same as the
distribution of $\tilde{X}$ if the latter were started at $y$\,.
Assuming that $x>y$, I have that until $\tilde{X}$ and $\tilde{X}'$ meet,
\begin{equation*}
  \tilde{X}_t-\tilde{X}'_t\le x-y+2\int_0^t\sqrt{\lambda_s+\theta_s (\breve{q}_s-\kappa_s)^++\mu_s\,(\breve{q}_s\wedge \kappa_s)}\,d\tilde{W}_s\,.
\end{equation*}
For $\tau_{x,y}=\inf\{t:\,\tilde{X}_t=\tilde{X}'_t\}$, 
\begin{equation*}
  \mathbf{P}(\tau_{x,y}>t)\le
  \mathbf{P}(2\inf_{u\in[0,t]}\int_0^u\sqrt{\lambda_s+\theta_s
    (\breve{q}_s-\kappa_s)^++\mu_s\,(\breve{q}_s\wedge
    \kappa_s)}\,d\tilde{W}_s\ge y-x)\,.
\end{equation*}
Let $\Gamma(u)=\inf\{v:\,\int_0^v\bl(\lambda_s+\theta_s
    (\breve{q}_s-\kappa_s)^++\mu_s\,(\breve{q}_s\wedge
    \kappa_s)\br)\,ds=u\}$\,, which is finite by 
\eqref{eq:17}. Since the processes 
$(\int_0^u\sqrt{\lambda_s+\theta_s
    (\breve{q}_s-\kappa_s)^++\mu_s\,(\breve{q}_s\wedge
    \kappa_s)}\,d\tilde{W}_s,\,u\in\R_+)$ and 
$(\tilde{W}_{\int_0^u(\lambda_s+\theta_s
    (\breve{q}_s-\kappa_s)^++\mu_s\,(\breve{q}_s\wedge
    \kappa_s))\,ds},\,u\in\R_+)$ have the same distribution,
$\Gamma^{-1}(u)=\int_0^u(\lambda_s+\theta_s
    (\breve{q}_s-\kappa_s)^++\mu_s\,(\breve{q}_s\wedge
    \kappa_s))\,ds$, the random variables $(-\inf_{u\in[0,\Gamma^{-1}(t)]}
\tilde{W}_u)$ and $\abs{\tilde{W}_{\Gamma^{-1}(t)}}$ have the same distribution,
and the random variables $\tilde{W}_{\Gamma^{-1}(t)}$ and 
$\sqrt{\Gamma^{-1}(t)}\tilde{W}_{1}$ have the same distribution,
I conclude that
\begin{multline*}
  \mathbf{P}(2\inf_{u\in[0,t]}\int_0^u\sqrt{\lambda_s+\theta_s
    (\breve{q}_s-\kappa_s)^++\mu_s\,(\breve{q}_s\wedge
    \kappa_s)}\,d\tilde{W}_s\ge y-x)=\mathbf{P}(2\inf_{u\in[0,\Gamma^{-1}(t)]}
\tilde{W}_u\ge y-x)\\=\mathbf{P}(2\abs{\tilde{W}_{\Gamma^{-1}(t)}}\le x-y)=
\mathbf{P}(2\abs{\tilde{W}_1}\le (x-y)/\sqrt{\Gamma^{-1}(t)})\to0\text{ as $t\to\infty$},
\end{multline*}
It follows that 
$  \mathbf{P}(\tau_{x,y}>t)\to0$ as $t\to\infty$\,. Furthermore, the latter
convergence is uniform over $x$ and $y$ from  bounded sets.
Therefore, the $T$-periodic version of $\tilde{X}$ is unique in distribution
and one has convergence in the distance of total variation 
to the distribution of that process from an arbitrary
initial distribution. (For a sample argument, let $\nu$ denote the
distribution of 
$\breve{X}_0$, let $\tilde{\nu}$ denote a probability distribution on
$\R$,
  and let $\nu_{x,\ell}$ denote the distribution of
$(\tilde{X}_{\ell T+s},\,s\in\R_+)$ with  $\tilde{X}_0=x$, where
$\ell\in\Z_+$.
 Then, for a bounded measurable function
$f$ on $\D(\R_+,\R)$ and $U\in\R_+$,
$\abs{\int_\R\int_{\D(\R_+,\R)}f(z)\nu_{x,\ell}(dz)\tilde{\nu}(dx) -
\int_\R  \int_{\D(\R_+,\R)}f(z)\nu_{x,0}(dz)\nu(dx)}\le
2 \sup_{\abs{x}\le U}
\abs{\int_{\D(\R_+,\R)}f(z)\nu_{x,\ell}(dz)
  -\int_{\D(\R_+,\R)}f(z)\nu_{0,\ell}(dz)}+
2\sup_{x\in\R}\abs{f(x)}(\tilde{\nu}(x:\,\abs{x}>U)+\nu(x:\,\abs{x}>U))\le
2\sup_{x\in\R}\abs{f(x)}\bl(
 \sup_{\abs{x}\le U}\mathbf{P}(\tau_{x,0}>\ell T)
+
\tilde{\nu}(\abs{x}>U)+\nu(\abs{x}>U)\br)\,. 
$ The latter expression converges to zero as $\ell\to\infty$ and $U\to\infty$\,.)

\end{proof}
\begin{remark}
If the functions $(\theta_t,\,t\in\R_+)$ and $(\mu_t,\,t\in\R_+)$ are
bounded, then the existence of the
  periodic version of $\tilde{X}$ can be deduced from Theorem 5.2 on
  p.90 of Has'minskii \cite{Has80} (with $V(t,x)=x^2$). I haven't found other results in
  the literature which directly apply, the sticking point being that the
  equation coefficients are not differentiable functions of time and space.
\end{remark}

\section{Convergence of stationary distributions}
\label{sec:stationary-solutions}
In this section I will assume constant arrival, service, and abandonment
rates, so  
$\lambda^n_t=\lambda^n\ge0$,  $\theta^n_t=\theta^n>0$, 
 $\mu^n_t=\mu^n>0$, $\lambda_t=\lambda\ge0$,
$\theta_t=\theta>0$, 
 $\mu_t=\mu>0$,  $\alpha_t=\alpha$, $\beta_t=\beta$,
and  $\gamma_t=\gamma$.
The number of servers $K^n_t$ is also assumed to be constant
 which I will take as the scaling parameter $n$, 
so $\kappa_t=1$. Accordingly, $\delta_t=0$.
The equations for the fluid- and diffusion-scale limits which appear
in Theorem~\ref{the:fluid} and Theorem~\ref{the:diff} assume the
following form:
\begin{align}
    \label{eq:11}
     q_t&=q_0+\lambda t
-\int_0^t\theta (q_s-1)^+\,ds-\int_0^t\mu\, (q_s\wedge 1)\,ds,\\
      \label{eq:10}
  X_t&=X_0+\int_0^t\bl(\alpha-\gamma(q_s-1)^+-\beta(q_s\wedge1)\br)\,ds -\int_0^t\theta
\bl(\ind{q_s>1}X_s+\ind{q_s=1} X_s^+\br)\,ds
\nonumber\\&-\int_0^t\mu
\bl(\ind{q_s<1}X_s+\ind{q_s=1}X_s\wedge 0\br)\,ds+
\int_0^t\sqrt{\lambda+\theta (q_s-1)^++\mu\,(q_s\wedge 1)}\,dW_s\,.
\end{align}
First, I investigate stationary solutions of \eqref{eq:11}.
\begin{lemma}
  \label{le:stat-q}
If $\lambda\ge\mu$, then
$\lim_{t\to\infty}q_t=(\lambda-\mu)/\theta+1$. 
If $\lambda\le\mu$, then
$\lim_{t\to\infty}q_t=\lambda/\mu$.
For all $t$, $q_t\not=1$ except when $q_0=1$ and $\lambda=\mu$
in which case $q_t=1$ for all $t$\,.
 \end{lemma}
\begin{proof}
Suppose that $\lambda> \mu$\,.
Then by (\ref{eq:11}),
\begin{equation*}
  \frac{d}{dt}\,\bl(q_t-\frac{\lambda-\mu}{\theta}-1\br)^2=
2\bl(q_t-\frac{\lambda-\mu}{\theta}-1\br)
(\lambda-\theta (q_t-1)^+-\mu\,(q_t\wedge 1))\,.
\end{equation*}
If $q_t-(\lambda-\mu)/\theta-1\ge\epsilon$ for $\epsilon>0$, then 
$\lambda-\theta (q_t-1)^+-\mu\,(q_t\wedge 1)=\lambda-\theta
(q_t-1)-\mu \le 
-\theta \epsilon$\,. 
If $q_t-(\lambda-\mu)/\theta-1\le-\epsilon$ 
for $\epsilon\in(0,(\lambda-\mu)/\theta)$, then 
$\lambda-\theta (q_t-1)^+-\mu\,(q_t\wedge 1)
\ge \lambda-\theta (q_t-1)^+-\mu\ge \theta \epsilon$\,. 
Hence, $q_t\to (\lambda-\mu)/\theta+1$ as $t\to\infty$\,.

If $\lambda< \mu$, then a similar reasoning applied  to the
function
$(q_t-\lambda/\mu)^2$ shows that $q_t\to
\lambda/\mu$\,.
Suppose $\lambda=\mu$\,. Then $\dot{q}_t=\mu(1-q_t)^+-\theta(q_t-1)^+$\,.
Hence,
$(d/dt)(q_t-1)^2=2(q_t-1)\bl(\mu(1-q_t)^+-\theta(q_t-1)^+\br)\le
-2(\mu\wedge\theta)(q_t-1)^2$\,. Consequently, $q_t\to 1$\,.

\end{proof}
The Markov chain  $Q^n$ is a birth-and-death process on $\mathbb{Z}_+$
with birth rates $\lambda^n$ and death rates $\mu^n (i\wedge n)+
\theta^n (i-n)^+$\,. Since $\sum_{k=1}^\infty(\lambda^n)^k/
\prod_{i=1}^k\bl(\mu^n(i\wedge n)+
\theta^n (i-n)^+\br)<\infty$,  
it admits a unique stationary distribution which is a limit in 
the distance of total
variation of the transient distributions  for any initial condition.
Let $\hat{Q}^n=(\hat{Q}^n_t,\,t\in\R_+)$ represent the stationary
version of $Q^n$ and let $\hat{q}_0=\lim_{t\to\infty}q_t$\,. 
\begin{theorem}
  \label{the:sta_flui}
Suppose that $\lambda^n/n\to\lambda$, that $\mu^n\to\mu$, and that
$\theta^n\to\theta$ as $n\to\infty$\,. Then, for all $\epsilon>0$
and $L>0$,
\begin{equation*}
  \lim_{n\to\infty}\mathbf{P}(\sup_{t\in[0,L]}\abs{\frac{\hat{Q}^n_t}{n}-\hat{q}_0}>\epsilon)=0\,.
\end{equation*}
\end{theorem}
\begin{proof}
  By part 2(b) of Lemma~\ref{le:bound}, $\sup_{n\in\N}\mathbf{E}
\hat{Q}^n_0/n<\infty$, so 
the sequence $\hat{Q}^n_0/n$ is tight. By Theorem~\ref{the:fluid}, the
sequence of processes $(\hat{Q}^n_t/n,\,t\in\R_+)$ is tight and any
limit point $(q_t,\,t\in\R_+)$ 
for convergence in distribution is  a solution to 
(\ref{eq:11})  for a suitable $\R_+$-valued random variable $q_0$ where,
by Fatou's lemma, $\mathbf{E}q_0<\infty$\,. Since
$(\hat{Q}^n_t/n,\,t\in\R_+)$ is stationary, so is
$(q_t,\,t\in\R_+)$\,. By the proof of
 Lemma~\ref{le:stat-q}, $\abs{q_t-
\hat{q}_0}$  decreases in $t$ and tends to zero
 as $t\to\infty$\,, so by dominated convergence $\mathbf{E}\abs{q_t-
\hat{q}_0}\to 0$\,. By stationarity, $\mathbf{E}\abs{q_t-
\hat{q}_0}=0$\,.
\end{proof}

Let process $\hat{X}^n=(\hat{X}^n(t),\,t\in\R_+)$ represent the
stationary version of $X^n$\,,i.e., $\hat{X}^n(t)=\sqrt{n}(
\hat{Q}^n_t/n-\hat{q}_0)$\,.
\begin{theorem}
  \label{the:stat}
Suppose that $\sqrt{n}(\lambda^n/n-\lambda)\to \alpha$,
$\sqrt{n}(\mu^n-\mu)\to\beta$, and
$\sqrt{n}(\theta^n-\theta)\to \gamma$ as
$n\to\infty$.
Then the processes
$\hat{X}^n$ converge in distribution in $\D(\R_+,\R)$ as $n\to\infty$
to a
stationary continuous-path Markov process
$\hat{X}=(\hat{X}_t,\,t\in\R_+)$.

If  $\lambda<\mu$, then the
process $\hat{X}$ is Gaussian
with  $\mathbf{E}\hat{X}_t=\alpha/\mu-\beta\,\lambda/\mu^2$ and  
$\mathbf{Cov}\,(\hat{X}_u,\hat{X}_v)=(\lambda/\mu)e^{-\mu\abs{u-v}}$. If 
$\lambda>\mu$, then the process $\hat{X}$
is Gaussian 
with $\mathbf{E}\hat{X}_t=\alpha/\theta-\gamma(\lambda-\mu)/\theta^2-\beta/\theta$ and  
$\mathbf{Cov}\,(\hat{X}_u,\hat{X}_v)=(\lambda/\theta)e^{-\theta\abs{u-v}}$.
 If
$\lambda=\mu$, then 
\begin{equation*}
  \hat{X}_t=\hat{X}_0+(\alpha-\beta) t-\theta \int_0^t \hat{X}_s^+\,ds+\mu\int_0^t(- \hat{X}_s)^+\,ds+
\sqrt{2\mu}\hat{W}_t\,,
\end{equation*}
where the  distribution of
$\hat{X}_0$ has density
$C\exp\bl(\bl((\alpha-\beta) x
-(x^2/2)(\theta \ind{x\ge0}
+\mu\ind{x<0})\br)/\mu\br)$\,, $(\hat{W}_t,\,t\in\R_+)$ is a standard
Wiener process, and $\hat{X}_0$ and $(\hat{W}_t,\,t\in\R_+)$ are independent.
\end{theorem}
\begin{proof}
The distributions of the
random variables  $X^n_t$ with $Q^n_0=q_0=0$ converge in the distance of total variation as
$t\to\infty$ to the distribution of $\hat{X}^n_0$\,. By 
part 2(c) of Lemma \ref{le:bound} and Fatou's lemma,
$\sup_{n\in\N}\mathbf{E}(\hat{X}^n_0)^2<\infty$, so  the sequence of the
 distributions 
 of the $\hat{X}^n_0$ is tight. By Theorem~\ref{the:diff}, the
sequence of the distributions  of the $\hat{X}^n$
is tight and any limit point in distribution
$(\grave{X}_t,\,t\in\R_+)$ satisfies the equation
\begin{multline*}
  \grave{X}_t=\grave{X}_0+\bl(\alpha-\gamma(\hat{q}_0-1)^+-
\beta(\hat{q}_0\wedge1)\br)\,t-\int_0^t\theta\bl(\ind{\hat{q}_0>1}\grave{X}_s
+\ind{\hat{q}_0=1} \grave{X}_s^+\br)\,ds\\-\int_0^t\mu
\bl(\ind{\hat{q}_0<1}\grave{X}_s+\ind{\hat{q}_0=1}(\grave{X}_s\wedge 0)
\br)\,ds+
\sqrt{\lambda+\theta (\hat{q}_0-1)^++\mu\,(\hat{q}_0\wedge1)}\,\grave{W}_t,
\end{multline*}
where $(\grave{W}_t,\,t\in\R_+)$ is a standard Wiener process and 
$\grave{X}_0 $ and $(\grave{W}_t,\,t\in\R_+)$ are independent. Since
the $X^n$ are stationary, so is $\grave{X}$\,. Stationary
distributions of one-dimensional diffusions
  are available in the literature, see, e.g., Skorokhod \cite{Sko89}.
\end{proof}
\begin{remark}
  Fleming, Simon, and Stolyar \cite{FleSimSto94} obtain the
   distribution of $\hat{X}_t$ provided $\lambda=\mu$ starting with
an   explicit formula for the stationary distribution of $Q^n$\,.
\end{remark}
  \begin{remark}
    If $\lambda\le \mu$, then the condition that
    $\sqrt{n}(\theta^n-\theta)\to \gamma$ can be disposed of and one
    can merely require that $\theta^n\to\theta$, as in Theorem 2 in
    Garnett, Mandelbaum, and Reiman \cite{GarManRei02}.
  \end{remark}
\begin{remark}
    The limits obtained in Theorems 2.1 and 2.3 in Whitt \cite{Whi04}
    correspond to the case
    where $\lambda^n=n\lambda$, $\mu^n=\mu$, $\theta^n=\theta$, and $\lambda>\mu$, so $\alpha=\beta=\gamma=0$\,.
  \end{remark}
\appendix
\section{Appendix}
\label{sec:appendix}
\begin{lemma}
  \label{le:maj}
Let $(F(t),\,t\in\R_+)$  be a function of locally bounded variation and
$(f(t),\,t\in\R_+)$ be a locally bounded  Lebesgue
measurable function. If a locally integrable function
$(y(t),\,t\in\R_+)$ is such that 
$y(t)\prec F(t)-\int_0^t f(s)y(s)\,ds$, then
\begin{equation*}
  y(t)\le e^{-\int_0^t f(s)\,ds}F(0)+e^{-\int_0^t f(s)\,ds}\int_0^t e^{\int_0^s f(u)\,du}\,dF(s)\,.
\end{equation*}
\end{lemma}
\begin{proof}
  Let
  $g(t)=F(t)-\int_0^t f(s)y(s)\,ds-y(t)$\,. The function
  $(g(t))$ is nondecreasing, $g(0)\ge0$, and 
  \begin{equation*}
y(t)= F(t)-g(t)-\int_0^t f(s)y(s)\,ds\,.
  \end{equation*}
Hence,
\begin{multline*}
  y(t)=e^{-\int_0^t f(s)\,ds}(F(0)-g(0))+e^{-\int_0^t f(s)\,ds}\int_0^t
  e^{\int_0^s f(u)\,du}\,d(F(s)-g(s))\\
\le e^{-\int_0^t f(s)\,ds}F(0)+e^{-\int_0^t f(s)\,ds}\int_0^t
  e^{\int_0^s f(u)\,du}\,dF(s)\,.
\end{multline*}
\end{proof}
The next lemma provides the bounds
that have been   used for the analysis of large-time
behaviour.  Let $T>0$ and
$  \sigma^n_s=\abs{\alpha^n_s-\gamma^n_s(q_s-\kappa_s)^+-\beta^n_s(q_s\wedge\kappa_s)}
+\abs{(\theta^n_s-\mu^n_s)\delta^n_s}\,.$
Let me recall that $q_t$ is defined by \eqref{eq:14}. 
\begin{lemma}
  \label{le:bound}
  \begin{enumerate}
  \item 
    \begin{enumerate}
    \item 
If the functions 
 $(\lambda_t,\,t\in\R_+)$, $(\mu_t,\,t\in\R_+)$,
and $(\theta_t,\,t\in\R_+)$  are
$T$-periodic and
$\int_0^T(\mu_s\wedge \theta_s)\,ds>0$\,, then, for all  $t\in\R_+$,
\begin{equation*}
q_t\le  e^{-\lfloor t/T\rfloor\int_0^T (\mu_s\wedge\theta_s)\,ds}\, q_0+\frac{ e^{\int_0^{T}
  (\mu_s\wedge\theta_s)\,ds}}{1-e^{-\int_0^{T}
  (\mu_s\wedge\theta_s)\,ds}}\int_{0}^{T}\lambda_s\,ds\,.
\end{equation*}
\item
If the functions 
$(\lambda^n_t,\,t\in\R_+)$,  $(\mu^n_t,\,t\in\R_+)$, and $(\theta^n_t,\,t\in\R_+)$ are
$T$-periodic and
$\int_0^T(\mu^n_s\wedge \theta^n_s)\,ds>0$\,, then,
for  all $t\in\R_+$ and $V>0$\,,
  \begin{equation*}
  \mathbf{E}Q^n_t\ind{Q^n_0\le
      V}\le e^{-\lfloor t/T\rfloor\int_0^T (\mu^n_s\wedge\theta^n_s)\,ds}   \,  \mathbf{E}Q^n_0\ind{Q^n_0\le
      V}+
\frac{ e^{\int_0^{T}
  (\mu^n_s\wedge\theta^n_s)\,ds}}{1-e^{-\int_0^{T}
  (\mu^n_s\wedge\theta^n_s)\,ds}}\int_{0}^{T}\lambda^n_s\,ds\,.
\end{equation*}\,.
\item
If the functions 
$(\lambda^n_t,\,t\in\R_+)$, $(\mu^n_t,\,t\in\R_+)$, 
 $(\theta^n_t,\,t\in\R_+)$, 
 $(\lambda_t,\,t\in\R_+)$,  $(\mu_t,\,t\in\R_+)$, and
 $(\theta_t,\,t\in\R_+)$ 
are
$T$-periodic, 
$\int_0^T(\mu_s\wedge \theta_s)\,ds>0$\,, and, for some $\epsilon>0$,
$\int_0^{T}\bl(2(\mu^n_s\wedge \theta^n_s)-
\epsilon\,\sigma^n_s-
(\mu^n_s\vee\theta^n_s)/(2\sqrt{n})\br)\,ds>0$,  
 then,  
for all $t\in\R_+$ and $V>0$, 
\begin{multline*}
    \mathbf{E}(X^n_t)^2\ind{\abs{X^n_0}\le V}\\\le
e^{-\lfloor t/T\rfloor\int_0^T \bl(2(\mu^n_s\wedge \theta^n_s)-\epsilon\,\sigma^n_s-
(\mu^n_s\vee\theta^n_s)/(2\sqrt{n})
\br)\,ds+\int_0^T \abs{2(\mu^n_s\wedge \theta^n_s)-\epsilon\,\sigma^n_s-
(\mu^n_s\vee\theta^n_s)/(2\sqrt{n})}\,ds}
\,\mathbf{E}(X^n_0)^2\ind{\abs{X^n_0}\le V} 
\\+
\frac{ e^{2\int_0^{T}\abs{2(\mu^n_s\wedge \theta^n_s)-\epsilon\,\sigma^n_s-
(\mu^n_s\vee\theta^n_s)/(2\sqrt{n})}\,ds}}{1-e^{-\int_0^{T}\bl(2(\mu^n_s\wedge \theta^n_s)-\epsilon\,\sigma^n_s-
(\mu^n_s\vee\theta^n_s)/(2\sqrt{n})\br)\,ds}}
\int_0^{T}\bl(\frac{1}{\epsilon}\,\sigma^n_s
+\frac{\lambda^n_s}{n}\\+
(\mu^n_s\vee\theta^n_s)\sup_{u\in\R_+}q_u+
\frac{1}{2\sqrt{n}}(\mu^n_s\vee\theta^n_s)\br)\,ds\,.
\end{multline*}
  \end{enumerate}
\item
  \begin{enumerate}
  \item 
If  $\sup_{t\in\R_+}\lambda_t<\infty$ and
  $\inf_{t\in\R_+}( \mu_t\wedge\theta_t)>0$, then, for all $t\in\R_+$,
  \begin{equation*}
    q_t\le
e^{-\int_0^t (\mu_s\wedge\theta_s)\,ds} q_0+ \frac{\sup_{s\in\R_+}\lambda_s}{\inf_{s\in\R_+}
(\theta_s\wedge \mu_s)}\,.
  \end{equation*}
\item
If  $\sup_{t\in\R_+}\lambda^n_t<\infty$ and
  $\inf_{t\in\R_+}( \mu^n_t\wedge\theta^n_t)>0$, then, 
 for all $t\in\R_+$ and $V>0$\,,
  \begin{equation*}
          \mathbf{E}Q^n_t\ind{Q^n_0\le
      V}\le e^{-\int_0^t (\mu^n_s\wedge\theta^n_s)\,ds}\, \mathbf{E}Q^n_0\ind{Q^n_0\le
      V}
+\frac{\sup_{s\in\R_+}\lambda^n_s}{\inf_{s\in\R_+}
(\theta^n_s\wedge \mu^n_s)}\,.
  \end{equation*}
\item 
If $\sup_{t\in\R_+}\lambda_t<\infty$,
  $\inf_{t\in\R_+}( \mu_t\wedge\theta_t)>0$,
$\sup_{t\in\R_+}\lambda^n_t<\infty$, 
$\sup_{t\in\R_+}\theta^n_t<\infty$,
$\sup_{t\in\R_+}\mu^n_t<\infty$, 
$\sup_{t\in\R_+}\sigma^n_t<\infty$,  and
$\inf_{t\in\R_+}\bl(2(\mu^n_t\wedge \theta^n_t)-\epsilon\,\sigma^n_t-
(\mu^n_t\vee\theta^n_t)/(2\sqrt{n})\br)>0$ for some  $\epsilon>0$\,, 
then,
 for all $t\in\R_+$ and $V>0$, 
\begin{multline*}
\mathbf{E}(X^n_t)^2\ind{\abs{X^n_0}\le V}\le
e^{-\int_0^t \bl(2(\mu^n_s\wedge \theta^n_s)-\epsilon\,\sigma^n_s-
(\mu^n_s\vee\theta^n_s)/(2\sqrt{n})
\br)\,ds}
\,\mathbf{E}(X^n_0)^2\ind{\abs{X^n_0}\le V} \\+
\dfrac{\sup_{s\in\R_+}\bl(\sigma^n_s/\epsilon
+\lambda^n_s/n+
(\mu^n_s\vee\theta^n_s)q_s+
(\mu^n_s\vee\theta^n_s)/(2\sqrt{n})\br)}{\inf_{s\in\R_+}\bl(2(\mu^n_s\wedge \theta^n_s)-\epsilon\,
\sigma^n_s
-(\mu^n_s\vee\theta^n_s)/(2\sqrt{n})
\br)}\,.
\end{multline*}
  \end{enumerate}
  \end{enumerate}

\end{lemma}
\begin{proof}
Let me start with part 1(b). 
   Since $Q^n_t\le Q^n_0+A^n_t$ by (\ref{eq:1}), I
have that $\mathbf{E}Q^n_t \ind{Q^n_0\le V}\le
V+\int_0^t\lambda^n_s\,ds$\,. By Lemma~\ref{le:mart}, the processes $(M^{n,A}_t\ind{Q^n_0\le
  V},\,t\in\R_+)$, $(M^{n,R}_t\ind{Q^n_0\le
  V},\,t\in\R_+)$, and $(M^{n,B}_t\ind{Q^n_0\le
  V},\,t\in\R_+)$ are $\mathbf{F}^n$-locally square integrable martingales with
respective predictable quadratic variation processes
$(\ind{Q^n_0\le V}\int_0^t\lambda^n_s\,ds\,,t\in\R_+)$, 
$(\ind{Q^n_0\le V}\int_0^t\theta^n_s\, (Q^n_s-K^n_s)^+\,ds,\,t\in\R_+)$, and
$(\ind{Q^n_0\le V}\int_0^t\mu^n_s\, (Q^n_s\wedge
K^n_s)\,ds,\,t\in\R_+)$\,. Since the latter processes are of finite
expectation, I obtain that $\mathbf{E}(M^{n,A}_t)^2\ind{Q^n_0\le
  V}<\infty$, $\mathbf{E}(M^{n,R}_t)^2\ind{Q^n_0\le
  V}<\infty$, and  $\mathbf{E}(M^{n,B}_t)^2\ind{Q^n_0\le
  V}<\infty$\,. In particular, $(M^{n,A}_t\ind{Q^n_0\le
  V},\,t\in\R_+)$, $(M^{n,R}_t\ind{Q^n_0\le
  V},\,t\in\R_+)$, and $(M^{n,B}_t\ind{Q^n_0\le
  V},\,t\in\R_+)$ are martingales\,. By \eqref{eq:2},
\begin{equation}
  \label{eq:34}
   \mathbf{E}Q^n_t\ind{Q^n_0\le V}\prec \mathbf{E}Q^n_0\ind{Q^n_0\le V}+
\int_0^t\lambda^n_s\,ds-\int_0^t(\mu^n_s\wedge\theta^n_s)
\,\mathbf{E}Q^n_s\ind{Q^n_0\le V}\,ds\,.
\end{equation}
By Lemma~\ref{le:maj},
\begin{equation}
  \label{eq:32}
    \mathbf{E}Q^n_t\ind{Q^n_0\le
      V}
\le
e^{-\int_0^t (\mu^n_s\wedge\theta^n_s)\,ds}     \mathbf{E}Q^n_0\ind{Q^n_0\le
      V}
+e^{-\int_0^t (\mu^n_s\wedge\theta^n_s)\,ds}\int_0^t e^{\int_0^s
  (\mu^n_u\wedge\theta^n_u)\,du}\lambda^n_s\,ds\,.
\end{equation}
If $vT\le t<(v+1)T$, where $v\in\Z_+$, then by $T$-periodicity,
\begin{multline*}
  e^{-\int_0^t (\mu^n_s\wedge\theta^n_s)\,ds}\int_0^t e^{\int_0^s
  (\mu^n_u\wedge\theta^n_u)\,du}\lambda^n_s\,ds\le
e^{-\int_0^{vT} (\mu^n_s\wedge\theta^n_s)\,ds}\sum_{i=1}^{v+1}
 e^{\int_0^{iT}
  (\mu^n_u\wedge\theta^n_u)\,du}\int_{(i-1)T}^{iT}\lambda^n_s\,ds\\
=e^{-v\int_0^{T} (\mu^n_s\wedge\theta^n_s)\,ds}\sum_{i=1}^{v+1}
 e^{i\int_0^{T}
  (\mu^n_u\wedge\theta^n_u)\,du}\int_{0}^{T}\lambda^n_s\,ds
=e^{-(v-1)\int_0^{T} (\mu^n_s\wedge\theta^n_s)\,ds}
\frac{ e^{(v+1)\int_0^{T}
  (\mu^n_u\wedge\theta^n_u)\,du}-1}{e^{\int_0^{T}
  (\mu^n_u\wedge\theta^n_u)\,du}-1}\int_{0}^{T}\lambda^n_s\,ds\\
\le\frac{ e^{2\int_0^{T}
  (\mu^n_u\wedge\theta^n_u)\,du}}{e^{\int_0^{T}
  (\mu^n_u\wedge\theta^n_u)\,du}-1}\int_{0}^{T}\lambda^n_s\,ds\,.
\end{multline*}
In addition, 
$e^{-\int_0^t (\mu^n_s\wedge\theta^n_s)\,ds}\le
e^{-\lfloor t/T\rfloor\int_0^T (\mu^n_s\wedge\theta^n_s)\,ds}$\,. Part
1(b) has been proved.

Part 1(a) follows by a similar argument if one observes that
by \eqref{eq:14},
\begin{equation*}
    q_t\prec q_0+\int_0^t\lambda_s\,ds 
-\int_0^t(\theta_s\wedge \mu_s)q_s\,ds\,
\end{equation*}
so that, by Lemma~\ref{le:maj},
\begin{equation}
  \label{eq:33}
q_t\le e^{-\int_0^t  (\theta_s\wedge \mu_s)\,ds}q_0+
e^{-\int_0^t  (\theta_s\wedge \mu_s)\,ds}
\int_0^te^{\int_0^s  (\theta_u\wedge \mu_u)\,du}\,\lambda_s\,ds\,.
\end{equation}
In order to prove part 1(c), let me note that by
 (\ref{eq:7}), 
\begin{multline*}
      (X^n_t)^2=(X^n_0)^2+
2\int_0^t X^n_{s-}\,dX^n_s+\sum_{0<s\le t}(\Delta X^n_s)^2=
(X^n_0)^2+
2\int_0^t\alpha^n_s X^n_{s}\,ds\\
-2\int_0^t\theta^n_sX^n_{s}\Bl(\bl(X^n_s-\delta^n_s+ \sqrt{n}( q_s-\kappa_s)\br)^+
-\sqrt{n}(q_s-\kappa_s)^+\Br)\,ds
-2\int_0^t\gamma^n_s(q_s-\kappa_s)^+X^n_s\,ds\\-2
\int_0^t\mu^n_sX^n_{s}\Bl(\bl(X^n_s+ \sqrt{n} q_s\br)\wedge(\delta^n_s+ \sqrt{n}\kappa_s) -
\sqrt{n}\,(q_s\wedge\kappa_s)\Br)\,ds-2\int_0^t\beta^n_s(q_s\wedge\kappa_s)X^n_s\,ds
 \\
+\frac{2}{\sqrt{n}}\,\int_0^t X^n_{s-}\,dM^{n,A}_s-
\frac{2}{\sqrt{n}}\,\int_0^t X^n_{s-}\,dM^{n,R}_s
-\frac{2}{\sqrt{n}}\,\int_0^t X^n_{s-}\,dM^{n,B}_s
\\+\frac{1}{n}\,\sum_{0<s\le t}(\Delta M^{n,A}_s)^2+
\frac{1}{n}\,\sum_{0<s\le t}(\Delta M^{n,R}_s)^2+
\frac{1}{n}\,\sum_{0<s\le t}(\Delta M^{n,B}_s)^2\,.
\end{multline*}
On noting that 
\begin{multline*}
  \theta^n_sX^n_{s}\Bl(\bl(X^n_s-\delta^n_s+ \sqrt{n}( q_s-\kappa_s)\br)^+
-\sqrt{n}(q_s-\kappa_s)^+\Br)+\mu^n_sX^n_{s}\Bl(\bl(X^n_s+ \sqrt{n} q_s\br)\wedge(\delta^n_s+ \sqrt{n}\kappa_s) -
\sqrt{n}\,(q_s\wedge\kappa_s)\Br)\\
\ge (\mu^n_s\wedge \theta^n_s)(X^n_s)^2- \abs{(\theta^n_s-\mu^n_s)\delta^n_sX^n_s}\,,
\end{multline*}
I obtain that, for $\epsilon>0$,
\begin{multline*}
(X^n_t)^2 
\prec (X^n_0)^2+
2\int_0^t\sigma^n_s\abs{ X^n_{s}}\,ds
-2\int_0^t( \mu^n_s\wedge\theta^n_s)(X^n_{s})^2\,ds
+\frac{2}{\sqrt{n}}\,\int_0^t X^n_{s-}\,dM^{n,A}_s-
\frac{2}{\sqrt{n}}\,\int_0^t X^n_{s-}\,dM^{n,R}_s\\
-\frac{2}{\sqrt{n}}\,\int_0^t X^n_{s-}\,dM^{n,B}_s
+\frac{1}{n}\,\sum_{0<s\le t}(\Delta M^{n,A}_s)^2+
\frac{1}{n}\,\sum_{0<s\le t}(\Delta M^{n,R}_s)^2+
\frac{1}{n}\,\sum_{0<s\le t}(\Delta M^{n,B}_s)^2\\ \prec (X^n_0)^2+
\frac{1}{\epsilon}\,\int_0^t\sigma^n_s\,ds
-\int_0^t\bl(2(\mu^n_s\wedge
\theta^n_s)-\epsilon\sigma^n_s\br)
(X^n_{s})^2\,ds
+\frac{2}{\sqrt{n}}\,\int_0^t X^n_{s-}\,dM^{n,A}_s-
\frac{2}{\sqrt{n}}\,\int_0^t X^n_{s-}\,dM^{n,R}_s\\
-\frac{2}{\sqrt{n}}\,\int_0^t X^n_{s-}\,dM^{n,B}_s
+\frac{1}{n}\,\sum_{0<s\le t}(\Delta M^{n,A}_s)^2+
\frac{1}{n}\,\sum_{0<s\le t}(\Delta M^{n,R}_s)^2+
\frac{1}{n}\,\sum_{0<s\le t}(\Delta M^{n,B}_s)^2\,.    
\end{multline*}

Hence, for $V>0$,
\begin{multline}
  \label{eq:9}
    (X^n_t)^2\ind{\abs{X^n_0}\le V}\prec (X^n_0)^2\ind{\abs{X^n_0}\le V}+
\frac{1}{\epsilon}\int_0^t\sigma^n_s\ind{\abs{X^n_0}\le V}\,ds
-\int_0^t\bl(2(\mu^n_s\wedge \theta^n_s)-\epsilon\,\sigma^n_s\br)(X^n_{s})^2\ind{\abs{X^n_0}\le V}\,ds\\
+\frac{2}{\sqrt{n}}\,\int_0^t X^n_{s-}\,
dM^{n,A,V}_s-
\frac{2}{\sqrt{n}}\,\int_0^t X^n_{s-}\,
\,dM^{n,R,V}_s
-\frac{2}{\sqrt{n}}\,\int_0^t X^n_{s-}\,
\,dM^{n,B,V}_s\\
+\frac{1}{n}\,\sum_{0<s\le t}(\Delta M^{n,A,V}_s)^2+
\frac{1}{n}\,\sum_{0<s\le t}(\Delta M^{n,R,V}_s)^2+
\frac{1}{n}\,\sum_{0<s\le t}(\Delta M^{n,B,V}_s)^2\,\,    
\end{multline}
where $M^{n,i,V}_s=M^{n,i}_s\ind{\abs{X^n_0}\le V}$, for $i=A,R,B$\,.

By Lemma~\ref{le:mart}, the processes $M^{n,i,V}=(M^{n,i,V}_t,\,t\in\R_+)$ are
$\mathbf{F}^n$-locally square integrable martingales with predictable quadratic variation
processes $\langle M^{n,i,V}\rangle=\langle
M^{n,i}\rangle\ind{\abs{X^n_0}\le V}$\,. 
Since $Q^n_t\le Q^ n_0+A^n_t$ by (\ref{eq:1}) and
$\mathbf{E}(A^n_t)^2=\int_0^t\lambda^n_s\,ds+\bl(\int_0^t\lambda^n_s\,ds\br)^2<\infty$, I have that 
$\mathbf{E}(Q^n_t)^2\ind{\abs{X^n_0}\le V}<\infty$\,. Hence, by (\ref{eq:3a}),
(\ref{eq:3b}), 
(\ref{eq:3c}),
$\mathbf{E}\langle M^{n,i,V}\rangle_t<\infty$, which implies that
$\mathbf{E}\bl(\sup_{s\le t}(M^{n,i,V}_s)^2\br)<\infty$,
that the $M^{n,i,V}$ are $\mathbf{F}^n$-martingales, and that
$\mathbf{E}( M^{n,i,V}_t)^2=\mathbf{E}\langle M^{n,i,V}\rangle_t$\,.
Consequently, the processes
$\bl(\int_0^{t}
 X^n_{s-}\,
\,dM^{n,i,V}_s,\,t\in\R_+\br)$
are $\mathbf{F}^n$-martingales.
 Since the $M^{n,i,V}$ are purely
discontinuous locally square integrable martingales by being of locally
bounded variation,
$\mathbf{E}\sum_{0<s\le t}(\Delta M^{n,i,V}_s)^2=
\mathbf{E}\langle M^{n,i,V}\rangle_{t}$\,.

On  taking expectations in (\ref{eq:9}),
\begin{multline*}
  \mathbf{E}(X^n_{t})^2\ind{\abs{X^n_0}\le V}\prec
  \mathbf{E}(X^n_{0})^2\ind{\abs{X^n_0}\le V} +
\frac{1}{\epsilon}\int_0^t\sigma^n_s\,ds
+\frac{1}{n}\,\bl(\mathbf{E}\langle M^{n,A}\rangle_{t}\ind{\abs{X^n_0}\le V}\\
+\mathbf{E}\langle M^{n,R}\rangle_{t}\ind{\abs{X^n_0}\le V}
+\mathbf{E}\langle M^{n,B}\rangle_{t}\ind{\abs{X^n_0}\le V}\br)
-
\int_0^t\bl(2(\mu^n_s\wedge \theta^n_s)-\epsilon\,\sigma^n_s\br)
\mathbf{E}(X^n_{s})^2\ind{\abs{X^n_0}\le V}\,ds\,.
\end{multline*}
By (\ref{eq:3a}), (\ref{eq:3b}),  (\ref{eq:3c}), and (\ref{eq:6}),
\begin{multline*}
\frac{1}{n}\,\bl(\mathbf{E}\langle M^{n,A}\rangle_{t}\ind{\abs{X^n_0}\le V}
+\mathbf{E}\langle M^{n,R}\rangle_{t}\ind{\abs{X^n_0}\le V}
+\mathbf{E}\langle M^{n,B}\rangle_{t}\ind{\abs{X^n_0}\le V}\br)\\
\prec
\int_0^t\frac{\lambda^n_s}{n}\,ds
+\int_0^{t}(\mu^n_s\vee\theta^n_s)\frac{\mathbf{E}Q^n_s\,\ind{\abs{X^n_0}\le V}\,}{n}
ds\prec
\int_0^t\frac{\lambda^n_s}{n}\,ds
+\int_0^{t}(\mu^n_s\vee\theta^n_s)\bl(\frac{\mathbf{E}X^n_s\ind{\abs{X^n_0}\le V}
}{\sqrt{n}}+q_s\br)\,
ds
\\\prec
\int_0^t\frac{\lambda^n_s}{n}\,ds+
\int_0^{t}(\mu^n_s\vee\theta^n_s)\bl(q_s+
\frac{1}{2\sqrt{n}}\br)\,ds
+\frac{1}{2\sqrt{n}}\int_0^{t}(\mu^n_s\vee\theta^n_s)\,\mathbf{E}(X^n_s)^2\ind{\abs{X^n_0}\le V}
\,ds \,.
\end{multline*}
Thus, for  $t\in\R_+$, 
  \begin{multline*}
      \mathbf{E}(X^n_{t})^2\ind{\abs{X^n_0}\le V}\prec
  \mathbf{E}(X^n_0)^2\ind{\abs{X^n_0}\le V}+
\frac{1}{\epsilon}\int_0^t\sigma^n_s\,ds
+\int_0^t\frac{\lambda^n_s}{n}\,ds+
\int_0^{t}(\mu^n_s\vee\theta^n_s)\bl(q_s+
\frac{1}{2\sqrt{n}}\br)\,ds\\
-
\int_0^t\bl(2(\mu^n_s\wedge \theta^n_s)-\epsilon\,\sigma^n_s-
\frac{1}{2\sqrt{n}}(\mu^n_s\vee\theta^n_s)
\br)\mathbf{E}(X^n_{s})^2\ind{\abs{X^n_0}\le V}\,ds\,.
\end{multline*}
By   Lemma~\ref{le:maj},
\begin{multline}
    \label{eq:30}
      \mathbf{E}(X^n_t)^2\ind{\abs{X^n_0}\le V}\le 
e^{-\int_0^t \bl(2(\mu^n_s\wedge \theta^n_s)-\epsilon\,\sigma^n_s-
(\mu^n_s\vee\theta^n_s)/(2\sqrt{n})
\br)\,ds}
\mathbf{E}(X^n_0)^2\ind{\abs{X^n_0}\le V}\\+e^{-\int_0^t \bl(2(\mu^n_s\wedge \theta^n_s)-\epsilon\,\sigma^n_s-
(\mu^n_s\vee\theta^n_s)/(2\sqrt{n})
\br)\,ds}\int_0^te^{\int_0^s\bl(2(\mu^n_u\wedge \theta^n_u)-\epsilon\,\sigma^n_u-
(\mu^n_u\vee\theta^n_u)/(2\sqrt{n})\br)\,du}\bl(\frac{\sigma^n_s}{\epsilon}\,
+\frac{\lambda^n_s}{n}\\+
(\mu^n_s\vee\theta^n_s)\bl(q_s+
\frac{1}{2\sqrt{n}}\br)\br)\,ds\,.
\end{multline}
In analogy with the earlier argument,
 if $vT\le t<(v+1)T$, where $v\in\Z_+$,
recalling that $\sup_{u\in\R_+}q_u<\infty$ by
part 1(a),
\begin{multline*}
e^{-\int_0^t \bl(2(\mu^n_s\wedge \theta^n_s)-\epsilon\,\sigma^n_s-
(\mu^n_s\vee\theta^n_s)/(2\sqrt{n})
\br)\,ds}\int_0^te^{\int_0^s\bl(2(\mu^n_u\wedge \theta^n_u)-\epsilon\,\sigma^n_u-
(\mu^n_u\vee\theta^n_u)/(2\sqrt{n})\br)\,du}\bl(\frac{\,\sigma^n_s}{\epsilon}
+\frac{\lambda^n_s}{n}\\+
(\mu^n_s\vee\theta^n_s)\bl(q_s+
\frac{1}{2\sqrt{n}}\br)\br)\,ds\le
e^{-v\int_0^{T} \bl(2(\mu^n_s\wedge \theta^n_s)-\epsilon\,\sigma^n_s-
(\mu^n_s\vee\theta^n_s)/(2\sqrt{n})
\br)\,ds+2\int_0^{T} \abs{2(\mu^n_s\wedge \theta^n_s)-\epsilon\,\sigma^n_s-
(\mu^n_s\vee\theta^n_s)/(2\sqrt{n})}\,ds}\\
\sum_{i=0}^{v}e^{i\int_0^{T}\bl(2(\mu^n_u\wedge \theta^n_u)-\epsilon\,\sigma^n_u-
(\mu^n_u\vee\theta^n_u)/(2\sqrt{n})\br)\,du}\int_0^{T}\bl(\frac{1}{\epsilon}\,\sigma^n_s
+\frac{\lambda^n_s}{n}+
(\mu^n_s\vee\theta^n_s)\bl(\sup_{u\in\R_+}q_u+
\frac{1}{2\sqrt{n}}\br)\br)\,ds\\
\le
\frac{ e^{
2\int_0^{T}\abs{2(\mu^n_u\wedge \theta^n_u)-\epsilon\,\sigma^n_u-
(\mu^n_u\vee\theta^n_u)/(2\sqrt{n})}\,du}}{1-e^{-\int_0^{T}\bl(2(\mu^n_u\wedge \theta^n_u)-\epsilon\,\sigma^n_u-
(\mu^n_u\vee\theta^n_u)/(2\sqrt{n})\br)\,du}}
\int_0^{T}\bl(\frac{1}{\epsilon}\,\sigma^n_s
+\frac{\lambda^n_s}{n}+
(\mu^n_s\vee\theta^n_s)\bl(\sup_{u\in\R_+}q_u+
\frac{1}{2\sqrt{n}}\br)\br)\,ds\,,
\end{multline*}
where the last inequality uses the fact that 
$\int_0^{T}\bl(2(\mu^n_u\wedge \theta^n_u)-\epsilon\,\sigma^n_u-
(\mu^n_u\vee\theta^n_u)/(2\sqrt{n})\br)\,du>0$\,.
The latter expression furnishes  the required bound.
Part 1 has been proved.

The assertions of part 2 also follow from the respective inequalities
\eqref{eq:32}, \eqref{eq:33}, and \eqref{eq:30}\,.
For instance  part 2(b) is obtained by applying  the bound
\begin{equation*}
  \int_0^t e^{\int_0^s (\theta^n_u\wedge \mu^n_u)\,du}\,\lambda^n_s\,ds\le
  \frac{\sup_{t\in\R_+}\lambda^n_t}{\inf_{t\in\R_+}
(\theta^n_t\wedge \mu^n_t)}\,\bl(e^{\int_0^t (\theta^n_s\wedge \mu^n_s)\,ds}-1\br)\,.
\end{equation*}
 \end{proof}

\begin{acknowledgements}
I am thankful to the referees for the careful reading of the
manuscript and insightful comments.
\end{acknowledgements}

\def\cprime{$'$} \def\cprime{$'$} \def\cprime{$'$} \def\cprime{$'$}
  \def\cprime{$'$} \def\polhk#1{\setbox0=\hbox{#1}{\ooalign{\hidewidth
  \lower1.5ex\hbox{`}\hidewidth\crcr\unhbox0}}} \def\cprime{$'$}
  \def\cprime{$'$} \def\cprime{$'$}


\begin{thebibliography}{10}
\providecommand{\url}[1]{{#1}}
\providecommand{\urlprefix}{URL }
\expandafter\ifx\csname urlstyle\endcsname\relax
  \providecommand{\doi}[1]{DOI~\discretionary{}{}{}#1}\else
  \providecommand{\doi}{DOI~\discretionary{}{}{}\begingroup
  \urlstyle{rm}\Url}\fi

\bibitem{EthKur86}
Ethier, S.N., Kurtz, T.G.: Markov Processes. Characterization and Convergence.
\newblock Wiley (1986)

\bibitem{FayMalMen95}
Fayolle, G., Malyshev, V.A., Men{\cprime}shikov, M.V.: Topics in the
  constructive theory of countable {M}arkov chains.
\newblock Cambridge University Press, Cambridge (1995)

\bibitem{FleSimSto94}
Fleming, P., Simon, B., Stolyar, A.: Heavy traffic limit for a mobile phone
  system loss model.
\newblock In: Proc. 2nd Internat. Conf. Telecommunication Systems, Modeling,
  and Anal., pp. 158--176. Nashville, TN (1994)

\bibitem{GarManRei02}
Garnett, O., Mandelbaum, A., Reiman, M.: Designing a call center with impatient
  customers.
\newblock Manufacturing Service Oper. Management \textbf{4}(3), 208--227 (2002)

\bibitem{Has80}
Has'minskii, R.: Stochastic Stability of Differential Equations.
\newblock Sijthoff \& Noordhoff (1980).
\newblock (Original title: Ustoicivost' sistem differencial'nyh uravnenii pri
  slucainyh vozmusceniyah ih parametrov, Nauka, Moscow, 1969)

\bibitem{IkeWat}
Ikeda, N., Watanabe, S.: Stochastic Differential Equations and Diffusion
  Processes, 2nd edn.
\newblock North Holland (1989)

\bibitem{jacshir}
Jacod, J., Shiryaev, A.: Limit Theorems for Stochastic Processes,
  \emph{Grundlehren der Mathematischen Wissenschaften [Fundamental Principles
  of Mathematical Sciences]}, vol. 288, 
\newblock Springer-Verlag, Berlin (1987)

\bibitem{lipshir}
Liptser, R., Shiryayev, A.: Theory of Martingales.
\newblock Kluwer (1989)

\bibitem{ManMasReiRidSto02}
Mandelbaum, A., Massey, W., Reiman, M., Rider, B., Stolyar, A.: Queue lengths and
  waiting times for multiserver queues with abandonment and retrials.
\newblock Telecommunication Systems \textbf{21}(2-4),
 149--172, (2002)

\bibitem{ManMasRei98}
Mandelbaum, A., Massey, W.A., Reiman, M.: Strong approximations for {M}arkovian
  service networks.
\newblock Queueing Syst. \textbf{30}, 149--201 (1998)

\bibitem{MeyTwe93}
Meyn, S.P., Tweedie, R.L.: Markov chains and stochastic stability.
\newblock Communications and Control Engineering Series. Springer-Verlag London
  Ltd., London (1993)

\bibitem{Whi07}
Pang, G., Talreja, R., Whitt, W.: Martingale proofs of many-server
  heavy-traffic limits for {M}arkovian queues.
\newblock Probab. Surv. \textbf{4}, 193 -- 267 (electronic) (2007)

\bibitem{Puh07}
Puhalskii, A.A.: The {$M_t/M_t/K_t+M_t$} queue in heavy traffic.
\newblock Available at arXiv.org/abs/0807.4621

\bibitem{Sko89}
Skorokhod, A.V.: Asymptotic Methods in the Theory of Stochastic Differential
  Equations, \emph{Translations of Mathematical Monographs}, vol.~78.
\newblock American Mathematical Society, Providence, RI (1989)

\bibitem{MR838384}
Smorodinski{\u\i}, A.V.: Asymptotic distribution of queue length in a queueing
  system.
\newblock Avtomat. i Telemekh. (2), 92--99 (1986)

\bibitem{TalWhi09}
Talreja, R., Whitt, W.: Heavy-traffic limits for waiting times in many-server
  queues with abandonment.
\newblock Ann. Appl. Probab. \textbf{19}(6), 2137--2175 (2009).
\newblock \doi{10.1214/09-AAP606}.
\newblock
  \urlprefix\url{http://0-dx.doi.org.skyline.ucdenver.edu/10.1214/09-AAP606}

\bibitem{Whi04}
Whitt, W.: Efficiency-driven heavy-traffic approximations for many-server
  queues with abandonment.
\newblock Management Sci. \textbf{50}(10), 1449--1461 (2004)

\bibitem{Whi05}
Whitt, W.: Engineering solution of a basic call center model.
\newblock Management Sci. \textbf{51}(2), 221--235 (2005)

\bibitem{ZelMan05}
Zeltyn, S., Mandelbaum, A.: Call centers with impatient customers: many-server
  asymptotics of the {$M/M/n+G$} queue.
\newblock Queueing Syst. \textbf{51}(3-4), 361--402 (2005)

\end{thebibliography}
\end{document}